\documentclass[reqno]{amsart}

\usepackage[utf8]{inputenc}
\usepackage[T1]{fontenc}
\usepackage{amsmath,amsfonts,amssymb,amsthm}
\usepackage{thmtools} 
\usepackage{color}
\usepackage{textcmds}
\usepackage{subcaption}
\usepackage[european]{circuitikz}
\usepackage{pdflscape}
\usepackage{multirow}
\usepackage{hyperref}

\usepackage{multirow}
\usepackage{multicol}
\usepackage{bm}
\usepackage{enumitem}
\usepackage{graphicx} 
\usepackage{booktabs} 
\usepackage{multirow} 
\usepackage{makecell} 
\usepackage{mathtools}
\usepackage{csquotes}
\usepackage{tabularx}
\usepackage{nicefrac}
\newcolumntype{L}{>{\raggedright\arraybackslash}X}   
\newcolumntype{C}{>{\centering\arraybackslash}X}     
\newcolumntype{R}{>{\raggedleft\arraybackslash}X}    

\usepackage{graphicx}

\def\ex{\mbox{\scriptsize\text{ex}}}

\def\env{\mbox{\scriptsize\text{env}}}
\def\diff{\ensuremath \mathrm{d}}

\def\myref{\mbox{\scriptsize\text{ref}}}

\def\nnodes{n_e}
\def\resist{G}
\def\resistM{\bm{G}}
\def\resistT{G_T}
\def\resistN{G_N}
\def\resistTM{\bm{G}_T}
\def\resistNM{\bm{G}_N}
\def\capM{{\bm{C}}_C}
\def\nT{n_T}
\newcommand{\resistTi}[1]{G_{T#1}}
\def\canonic{\hat{\bm{e}}}
\def\nbuild{N}

\newcommand{\subsystem}[1]{\ensuremath\mathfrak{S}_{#1}} 
\def\system{\mathfrak{S}} 
\newcommand{\splitpart}[1]{\ensuremath\mathfrak{P}_{#1}} 

\usepackage{bbm}
\def\one{\mathbbm{1}}
\def\R{\mathbbm{R}}
\usepackage{booktabs}
\theoremstyle{plain}
\newtheorem{theorem}{Theorem}
\newtheorem{proposition}[theorem]{Proposition}
\newtheorem{lemma}[theorem]{Lemma}

\theoremstyle{definition}
\newtheorem{definition}[theorem]{Definition}
\newtheorem{remark}[theorem]{Remark}
\newtheorem{example}[theorem]{Example}

\usepackage{amsmath}
\usepackage{amssymb}
\usepackage{mathrsfs}
\usepackage{pgfplots}
\pgfplotsset{compat=1.18} 
\usepgfplotslibrary{groupplots} 
\begin{document}

\title[Energy-Consistent Splitting and Decomposition Approaches for pH-ODEs]{Energy-Consistent Splitting and Decomposition Approaches for Coupled port-Hamiltonian ODEs}

\author[M. Mönch et al.]{Marius Mönch$^{1,\star}$ \and Nicole Marheineke$^1$ \and  Andreas Bartel$^2$ \and Kevin Schäfers$^2$ \and Michael Günther$^2$}

\date{\today\\
$^1$ Trier University, Arbeitsgruppe Modellierung und Numerik, Universit\"atsring 15, D-54296 Trier, Germany\\
$^2$ University of Wuppertal, School of Mathematics and Natural Sciences, Gaußstr.~20, D-42119 Wuppertal, Germany \\
$^\star$ corresponding author, moench@uni-trier.de, orcid 0009-0000-2582-2199}

\begin{abstract}
Operator splitting provides an attractive approach for the numerical integration of (coupled) port-Hamiltonian systems, as it allows the underlying system structure to be exploited at the level of the individual subproblems. However, the choice of the decomposition is not unique and may strongly affect both the computational efficiency and the preservation of the energy behavior of the original system. In this work, we investigate this interplay systematically and introduce energy consistency as a criterion for assessing splitting methods for port-Hamiltonian ordinary differential equations. We derive sufficient conditions under which a splitting based on a given decomposition inherits the energy behavior of the continuous system and use these conditions to analyze several decomposition strategies for coupled port-Hamiltonian systems. In particular, we compare decompositions that preserve the structure with approaches that exploit lower-dimensional subsystem dynamics or separated time scales. The analysis is complemented by numerical experiments using Strang splitting and its multiple-time-stepping extension. A scalable electro-thermal benchmark with fast electrical and slow thermal dynamics is employed to assess accuracy, energy behavior, and computational efficiency. The results demonstrate that preserving the port-Hamiltonian structure of the subflows is essential for energy-consistent splitting, whereas decompositions that exploit subsystem structure or time-scale separation can provide substantial computational advantages. In particular, the time-scale decomposition yields significant efficiency gains for systems with pronounced multirate characteristics, while structure-destroying decompositions may lead to undesirable energy behavior.
\end{abstract}


\maketitle

\section{Introduction}

Port-Hamiltonian systems (PHS) provide a systematic framework for modeling interconnected physical systems in terms of energy storage, energy exchange, dissipation, and external interaction. Their inherent structure makes the energy balance explicit and provides a natural basis for describing qualitative properties such as passivity, dissipation, and energy conservation. This is particularly attractive for multiphysical and networked applications, where systems from different physical domains are coupled through energy exchange; see, e.g., \cite{bartel2024,hauschild2020,ponce2024,vanderschaft2014}.

In this work, we consider port-Hamiltonian systems of implicit ordinary differential equations (pH-ODEs) of the form
\begin{equation}\label{eq:pH-ODE}
    \begin{aligned}
       \system\colon \quad \bm{E}(\bm x)\dot{\bm{x}} &= \bigl(\bm{J}(\bm{x})-\bm{R}(\bm{x})\bigr)\bm{z}(\bm{x})+\bm{B}(\bm{x})\bm{u}(t) = \bm{f}(t,\bm{x}),
       \qquad \bm{x}(t_0)=\bm{x}_0,\\
       \bm{y} &= \bm{B}(\bm{x})^{\top}\bm{z}(\bm{x}),
    \end{aligned}
\end{equation}
where $\bm{J}(\bm{x}) = - \bm{J}(\bm{x})^\top$ skew-symmetric, $\bm{R}(\bm{x}) = \bm{R}(\bm{x})^\top \succeq \bm{0}$ symmetric positive semi-definite, $\bm{E}(\bm{x})$ regular and $\bm{E}(\bm{x})^\top \bm{z}(\bm{x}) = \nabla \mathcal{H}(\bm{x})$.  The Hamiltonian $\mathcal{H}$ represents the stored energy. Consequently, assuming sufficient regularity, the continuous dynamics satisfy the power balance
\begin{equation}\label{eq:power_balance}
\tfrac{\mathrm{d}}{\mathrm{d}t} \mathcal{H}(\bm{x}(t))
= - \bm{z}(\bm{x}(t))^{\top} \bm{R}(\bm{x}(t)) \bm{z}(\bm{x}(t))
    + \bm{z}(\bm{x}(t))^{\top} \bm{B}(\bm{x}(t))\bm{u}(t)
\le \bm{y}(t)^{\top} \bm{u}(t),
\end{equation}
which expresses the passivity of the system. In particular, in the absence of external input, i.e., $\bm{u}\equiv \bm{0}$, the stored energy is non-increasing.
In many applications, large port-Hamiltonian systems are assembled from several interacting subsystems. The resulting coupled systems inherit a block structure from the underlying physical interconnection and may exhibit substantially different dynamic time scales. This structure is beneficial for numerical simulation, since the dynamics may be decomposed into smaller or otherwise simpler subproblems.

Operator splitting is a natural approach for exploiting such structure. The basic idea is to decompose the vector field $ \bm{f}(t,\bm{x})=\sum_{i=1}^N\bm{f}^{[i]}(t,\bm{x})$ and to approximate the flow of the full system by compositions of the flows associated with the subproblems. The subflows can then be approximated by suitable numerical integration methods.  Splitting schemes are well established in numerical analysis; see, e.g., \cite{blanes2024, mclachlan2002}. 
For port-Hamiltonian systems, however, the decomposition is not merely a computational choice. It may determine whether the individual subproblems retain the structural properties responsible for the energy behavior of the original system. Different decompositions may lead to subproblems with different effective dimensions, linearity properties, stability behavior, and energy-related structures.  A decomposition that reduces the dimension of the subproblems can therefore be computationally attractive, while simultaneously destroying the port-Hamiltonian structure. Conversely, a structure-preserving decomposition may retain the desired energy behavior but offer less potential for dimension reduction, parallelization or exploitation of different time scales.

Classical operator splitting schemes of order $p\geq 3$ with real-valued coefficients necessarily involve negative step sizes, which can cause stability issues for irreversible systems such as PHS with dissipation, where $\bm R \neq \bm 0$, \cite{blanes2005}. See also structure-preserving integrators for dissipative systems based on reversible-irreversible splitting in \cite{shang2020}.
Structure-preserving commutator-based higher-order splitting schemes with positive coefficients have recently been developed for linear and certain subclasses of nonlinear port-Hamiltonian systems \cite{moench2025,moench2026}. These approaches are based on energy-associated decompositions as well as port-based decompositions. In particular, the energy-associated JR-decomposition \cite{frommer2026} provides a way of constructing subproblems whose individual dynamics are compatible with the energy structure. Related ideas have been transferred to index-one port-Hamiltonian differential-algebraic equations (pH-DAEs) for certain assignments of the energy parts in the constraints, \cite{bartel2025}. For coupled systems, decompositions based on the underlying subsystem structure (block structure) can reduce the effective dimension of the individual subproblems. In this setting, splitting provides an alternative to waveform relaxation (dynamic iteration) \cite{guenther2021,lelarasmee1982}, closely related to component-wise partitioning \cite{arnold2001,busch2012,kuebler2000}. Recent work on coupled linear PHS \cite{lorenz2025} exploits the block structure while retaining a port-Hamiltonian formulation. Hierarchical splitting \cite{schaefers2026} further demonstrates the flexibility of combining different decomposition principles. These developments highlight a fundamental trade-off in the choice of a decomposition. 
Preserving the port-Hamiltonian structure of each subflow is closely related to preserving passivity and the associated energy behavior, whereas decompositions aimed primarily at computational efficiency may produce subproblems that no longer possess this structure. For multiphysical systems with separated time scales, an additional possibility is to choose the decomposition specifically to enable multiple time stepping. In multiple-time-stepping approaches, fast subflows are resolved with smaller time steps while slow components are advanced on a larger macro time scale, see, e.g., \cite{biesiadecki1993, grubmueller1991}. 
Despite the developments, there is currently no systematic criterion that connects the structural properties of a chosen decomposition with the energy behavior of the resulting splitting method, while simultaneously accounting for computational efficiency. Such a criterion is particularly relevant for coupled port-Hamiltonian systems, where structure preservation, subsystem complexity, and time-scale separation may lead to competing design objectives.

The aim of this paper is to investigate this trade-off from the perspective of energy consistency. The concept of energy consistency has been previously introduced and discussed, e.g., in the context of symplectic integration with collocation methods \cite{kotyczka2019} and Petrov-Galerkin schemes \cite{egger2021,giesselmann2025}. In \cite{celledoni2017}, energy-preserving and passivity-consistent numerical discretizations of PHS are developed via discrete gradient and splitting methods. In this paper, rather than introducing a particular splitting decomposition, we use energy consistency as a common criterion for analyzing different decomposition strategies. This allows us to distinguish decompositions according to whether the resulting subflows preserve the structural mechanisms underlying the continuous power balance, and to relate these properties to the computational characteristics of the corresponding splitting methods.
The main contributions of this work are as follows. We formulate a concept of energy consistency for numerical integration schemes applied to pH-ODEs and derive sufficient conditions under which a splitting method based on a given decomposition is energy-consistent. We use these conditions to systematically analyze several decomposition strategies for (coupled) port-Hamiltonian systems, including energy-associated, port-based, subsystem-based, diagonal, and time-scale decompositions. Energy consistency yields a discrete version of the power balance. We quantify the approximation quality of a discrete power balance by the order of energy consistency $q$. We show that, for classical real-valued splitting schemes with $p\leq 2$, this order is determined by the consistency order, i.e., $q=p$, whereas commutator-based splitting may exhibit a reduced order due to the induced quadrature rule, i.e., $q\leq p$. Moreover, we discuss the influence of numerical subflow approximations. In particular, Gauss collocation schemes for quadratic Hamiltonians \cite{hairer2006} and discrete gradient methods \cite{gonzalez1996, kinon2026} are covered by the proposed framework.
The analysis is augmented by an assessment of the computational efficiency of splitting methods based on the different decomposition strategies, with particular emphasis on reduced effective subsystem dimensions, parallelization potential and the exploitation of separated time scales. Numerical experiments based on Strang splitting and its multiple-time-stepping extension are performed, using a scalable electro-thermal benchmark that retains a port-Hamiltonian formulation and combines fast electrical with slow thermal dynamics.

The remainder of the paper is organized as follows. Section~\ref{sec:energy_consistency} develops the concept of energy consistency for splitting methods. Section \ref{sec:decomposition} analyzes different decomposition strategies with respect to their structural properties, energy consistency, and computational implications, and discusses their hierarchical use. Section \ref{sec:numerical_results} presents numerical experiments for an electro-thermal circuit model, including the multiple-time-stepping setting. Conclusions and perspectives are given in Section \ref{sec:conclusion}. Appendix~\ref{app:DGM} and Appendix~\ref{app:thermal-electric-modeling} provide details on discrete gradient methods and the port-Hamiltonian modeling of the electro-thermal benchmark problem, respectively.

\section{Concept of Energy Consistency}\label{sec:energy_consistency}

In the spirit of geometric numerical integration, we require that the numerical solution obtained by a numerical integration scheme reflects the energy behavior of the exact solution of a pH-ODE through a discrete power balance. We formalize this demand on the scheme through the following concept of energy consistency.

\begin{definition}[Energy consistency]\label{def: energy-consistent}
A numerical integration scheme $\bm{\Psi}$ for the pH-ODE \eqref{eq:pH-ODE} is said to be \emph{energy-consistent} if for every initial value $\bm{x}_0 \in \R^n$ and for every step size $h>0$ the numerical approximation $\bm{x}_1 = \bm{\Psi}_{t_1,t_0}(\bm{x}_0)$, $t_1=t_0+h$, admits a decomposition 
$$ \mathcal{H}(\bm{x}_1) - \mathcal{H}(\bm{x}_0) = \mathcal{D}_h + \mathcal{S}_h $$
into a dissipated energy $\mathcal{D}_h$ and a supplied energy $\mathcal{S}_h$ such that
\begin{enumerate}
  \item[(i)] $\mathcal{D}_h \le 0$, with $\mathcal{D}_h = 0$ whenever $\bm{R} \equiv \bm{0}$, and $\mathcal{S}_h = 0$ whenever $\bm{B} \equiv \bm{0}$;
  \item[(ii)] $\mathcal{D}_h$ and $\mathcal{S}_h$ are consistent with the continuous-time power balance \eqref{eq:power_balance}, i.e.,
  $$ \lim_{h \to 0} \frac{\mathcal{D}_h}{h} = -\bm{z}(\bm{x}_0)^\top \bm{R}(\bm{x}_0)\, \bm{z}(\bm{x}_0),
    \qquad
    \lim_{h \to 0} \frac{\mathcal{S}_h}{h} = \bm{y}(t_0)^\top \bm{u}(t_0). $$
\end{enumerate}
    We call this numerical integration scheme $\bm{\Psi}$ \emph{energy-consistent of order $q$}, $q\in \mathbb{N}$, if it holds 
    \begin{align*}
        \mathcal{D}_h &= \int_{t_0}^{t_0 + h}\hspace*{-0.2cm} -\bm{z}(\bm{x}(\tau))^\top \bm{R}(\bm{x}(\tau)) \bm{z}(\bm{x}(\tau)) \, \diff \tau + \mathcal{O}(h^{q+1}), \quad \mathcal{S}_h = \int_{t_0}^{t_0 + h} \bm{y}(\tau)^\top \bm{u}(\tau) \, \diff \tau + \mathcal{O}(h^{q+1}),
    \end{align*}
    where $\bm x(t)=\bm{\varphi}_{t,t_0}(\bm x_0)$ denotes the exact solution of \eqref{eq:pH-ODE}. We denote the exact dissipated energy by $\mathcal{D}^\star$ and the exact supplied energy by $\mathcal{S}^\star$.
\end{definition}

\begin{remark}
If $\bm{R}=\bm{0}$, then $\mathcal{D}_h=0$ and energy consistency yields the discrete energy balance
$\mathcal{H}(\bm{x}_1)-\mathcal{H}(\bm{x}_0)=\mathcal{S}_h$.
Thus, the numerical scheme satisfies a discrete passivity relation with storage function $\mathcal{H}$ and discrete supply $\mathcal{S}_h$, providing a discrete counterpart of the continuous-time power balance \eqref{eq:power_balance} \cite{khalil2002}.

If $\bm{B}=\bm{0}$, then $\mathcal{S}_h=0$ and energy consistency gives
$\mathcal{H}(\bm{x}_1)-\mathcal{H}(\bm{x}_0)=\mathcal{D}_h\leq 0$.
Hence, the numerical scheme inherits the dissipative property of the continuous system with respect to the Hamiltonian. 
Moreover, if $\bm{x}^\star$ is an equilibrium of the continuous system, $\mathcal{H}-\mathcal{H}(\bm{x}^\star)$ is a Lyapunov function for $\bm{x}^\star$, and the numerical integration scheme preserves this equilibrium, i.e., $\bm \Psi_{t_1,t_0}(\bm{x}^\star)=\bm{x}^\star$ for every step size $h>0$, then the dissipative property implies Lyapunov stability of $\bm{x}^\star$ under the numerical scheme.

If $\bm{R}=\bm{0}$ and $\bm{B}=\bm{0}$, then $\mathcal{D}_h=\mathcal{S}_h=0$, and the method is conservative,
$\mathcal{H}(\bm{x}_1)-\mathcal{H}(\bm{x}_0)=0$.
\end{remark}

\subsection{Splitting Methods}

The core idea of splitting methods is to decompose the right-hand side of the dynamic system and to approximate the analytic flow $\bm{\varphi}$ of the overall system by composing the flows  $\bm{\varphi}^{[i]}$ of the corresponding subproblems.  A $s$-stage splitting scheme of consistency order $p$ for an autonomous system decomposed into two subproblems has then the form \cite{blanes2024, mclachlan2002}
\begin{equation*}
  \bm{\Psi}_{h}
  =
  \bm{\varphi}^{[2]}_{b_sh}
  \circ
  \bm{\varphi}^{[1]}_{a_sh}
  \circ \,\cdots \,\circ
  \bm{\varphi}^{[2]}_{b_1}
  \circ
  \bm{\varphi}^{[1]}_{a_1}, \qquad \qquad  \bm{\Psi}_h(\bm{x}_0)=\bm{\varphi}_h(\bm{x}_0)+\mathcal{O}(h^{p+1})
\end{equation*}
with suitably chosen step size coefficients $a_1,\,\ldots,\,a_s$ and $b_1,\,\ldots,\,b_s$. In particular, the scheme is consistent if $\sum_{j=1}^s a_j = \sum_{j=1}^s b_j = 1$ is satisfied. The most prominent schemes are the first-order Lie-Trotter splitting \cite{trotter1959} and the symmetric second-order Strang splitting \cite{strang1968} 
\begin{equation}\label{eq:Strang}
    \bm{\Psi}_h
    =
    \bm{\varphi}^{[i]}_{h/2}
    \circ
    \bm{\varphi}^{[j]}_{h}
    \circ
    \bm{\varphi}^{[i]}_{h/2},
    \qquad
    i,j\in\{1,2\},\quad i\neq j.
\end{equation}
For general systems, classical splitting methods of order \(p\geq 3\) with real-valued coefficients necessarily involve negative time steps, i.e., at least one $a_j$ and one $b_j$ are negative, \cite{blanes2005}. This may lead to stability issues for pH-ODEs with dissipation $\bm{R}\neq \bm{0}$ and may therefore impose step-size restrictions on $h$. To preserve the dissipativity for all $h> 0$, higher-order commutator-based schemes with positive coefficients have been developed, \cite{moench2025, moench2026}. The classical order conditions arising from the Baker--Campbell--Hausdorff formula are here fulfilled by the interplay of problem-specific, decomposition-dependent commutators and suitable adapted coefficients.

In this work, we consider non-autonomous systems and deal with subproblems of the form
\begin{equation}\label{eq: Decomposition}
      \splitpart{i}: \quad \bm{E}(\bm{x}) \dot{\bm{x}} = \bm{f}^{[i]}(t,\bm{x}),\,\, i=1,2,  \qquad  \quad  \bm{f}^{[1]}(t,\bm{x})+\bm{f}^{[2]}(t,\bm{x})=\bm{f}(t,\bm{x}).
\end{equation}
The explicit time-dependence is accounted for in the initialization times of the subflows so that the splitting scheme takes the form
\begin{equation}\label{eq:splitting_method}
  \bm{\Psi}_{t_0+h,t_0}
  =
  \bm{\varphi}^{[2]}_{t_0+hB_{s}, t_0+hB_{s-1}}
  \circ
  \bm{\varphi}^{[1]}_{t_0+hA_{s}, t_0+hA_{s-1}}
  \circ \,\cdots \,\circ
  \bm{\varphi}^{[2]}_{t_0+hB_1,t_0}
  \circ
  \bm{\varphi}^{[1]}_{t_0+hA_1,t_0}
\end{equation}
with cumulative step size coefficients $A_k=\sum_{\ell=1}^k a_\ell$ and $B_k=\sum_{\ell=1}^k b_\ell$ for $k=1,...,s$.

\begin{lemma} \label{lemma: energy-consistent splitting}
    Let $\system$ be a pH-ODE \eqref{eq:pH-ODE} with Hamiltonian $\mathcal{H}\in \mathcal{C}^2(\mathbb{R}^n,\mathbb{R})$, flow matrix function $\bm E \in \mathcal{C}^1(\mathbb{R}^n,\mathbb{R}^{n\times n})$, port function $\bm B \in \mathcal{C}^1(\mathbb{R}^n,\mathbb{R}^{n\times m})$, input $\bm u \in \mathcal{C}^1([t_0,T],\mathbb{R}^m)$ and $\bm{f}:[t_0,T]\times \mathbb{R}^n\rightarrow \mathbb{R}^n$ Lipschitz continuous in $\bm x$. A consistent splitting method $\bm{\Psi}$ \eqref{eq:splitting_method} based on a decomposition of $\system$ into two subproblems is energy-consistent if 
    \begin{enumerate}[label=\roman*)]
        \item both subproblems $\splitpart{i}$, $i=1,2$, are described by pH-ODEs  of the form 
        \begin{align*}
        \bm{E}(\bm{x}) \dot{\bm{x}} & = \big( \bm{J}_i(\bm{x}) - \bm{R}_i(\bm{x}) \big) \bm{z}(\bm{x}) + \bm{B}_i(\bm{x}) \bm{u}(t)= \bm{f}^{[i]}(t,\bm{x}), \qquad \bm{y}_i  = \bm{B}_i(\bm{x})^\top \bm{z}(\bm{x}),
        \end{align*}
        with $\bm{J}_i(\bm{x}) = - \bm{J}_i(\bm{x})^\top$,  $\bm{R}_i(\bm{x}) = \bm{R}_i(\bm{x})^\top \succeq \bm{0}$ as well as $\bm{A}(\bm{x})=\bm{A}_1(\bm{x})+\bm{A}_2(\bm{x})$ for all $\bm{x} \in \R^{n}$, for all $\bm{A}\in \{ \bm{J},\bm{R},\bm{B}\} $
        and the corresponding output $\bm{y}_i$; 
        \item the subproblem vector fields $\bm{f}^{[i]}$ are locally Lipschitz continuous in $\bm{x}$, and $\bm B_i$ continuously differentiable, $i=1,2$;
        \item given $\bm{B} = \bm{0}$, then $\bm{B}_1 =\bm{B}_2 =  \bm{0}$;
        \item given $\bm{R}_1 \neq \bm{0}$, then $a_j \geq 0$ for all $j$; 
        \item given $\bm{R}_2 \neq \bm{0}$, then $b_j \geq 0$ for all $j$.
    \end{enumerate}
\end{lemma}
\begin{proof}
Both subproblems $\splitpart{i}$ equipped with an initial value possess a unique solution satisfying a respective power balance in terms of $\mathcal{H}$, i.e.
$$ \tfrac{\diff}{\diff t} \mathcal{H}(\bm{x}(t)) = - \bm{z}(\bm{x}(t))^\top \bm{R}_i(\bm{x}(t)) \bm{z}(\bm{x}(t)) + \bm{y}_i(t)^\top \bm{u}(t) \leq \bm{y}_i(t)^\top \bm{u}(t). $$
For the given splitting method, the cumulative step size coefficients $ A_k$ and  $B_k$, $k=1,...,s$, fulfill $A_s=B_s=1$ because of the scheme's consistency. The intermediate time points $\tau_j$, $j=1,...,2s$, encountered during the time step from $t_0$ to $t_0+h$ are
    $$ \tau_j = \begin{cases}
        t_0 + h B_{\nicefrac{j}{2}}, & j \; \text{even}, \\
        t_0 + h A_{\nicefrac{(j+1)}{2}}, & j \; \text{odd},
    \end{cases} $$ 
with $\tau_{-1}=\tau_0=t_0$, where the subproblem $\splitpart{i}$ is linked to the index $j$ via 
  \begin{align*}
        i &= i(j) = \begin{cases}
            2, & j \; \text{even}, \\
            1, & j \; \text{odd}. \\
        \end{cases}
    \end{align*}
Let $\bm x_{\nicefrac{j}{2s}}$ denote the intermediate solution of the $j$-th sub-step from $\tau_{j-2}$ to $\tau_j$ initialized with $\bm{x}_{\nicefrac{j}{2s}}(\tau_{j-2})=\bm{x}_{\nicefrac{(j-1)}{2s}}(\tau_{j-1})$, i.e.,
$\bm x_{\nicefrac{j}{2s}}(\tau)= \bm{\varphi}^{[i(j)]}_{\tau,\tau_{j-2}}(\bm x_{\nicefrac{(j-1)}{2s}}(\tau_{j-1}))$.
It satisfies the integrated power balance 
    \begin{align*}
        \mathcal{H}(\bm{x}_{\nicefrac{j}{2s}}(\tau_j)) - \mathcal{H}(\bm{x}_{\nicefrac{(j-1)}{2s}}(\tau_{j-1})) &= \int_{\tau_{j-2}}^{\tau_j} \hspace*{-0.5cm}- \bm{z}( \bm{x}_{\nicefrac{j}{2s}}(\tau) )^\top \bm{R}_{i(j)}( \bm{x}_{\nicefrac{j}{2s}}(\tau)  ) \bm{z}( \bm{x}_{\nicefrac{j}{2s}}(\tau) )
        + \bm{y}_{i(j)}(\tau)^\top \bm{u}(\tau) \, \mathrm{d}\tau \\
         &\leq \int_{\tau_{j-2}}^{\tau_j} \bm{y}_{i(j)}(\tau)^\top \bm{u}(\tau) \, \mathrm{d}\tau,
    \end{align*}
since $\tau_{j-2} \le \tau_j$, if $\bm{R}_{i(j)} \neq \bm{0}$.
Hence, we obtain
\begin{align} \label{eq:splitting-method_power-balance} \nonumber
        \mathcal{H}(\bm{x}_1) - \mathcal{H}(\bm{x}_0) &= \sum_{j=1}^{2s} \mathcal{H}(\bm{x}_{\nicefrac{j}{2s}}(\tau_j)) - \mathcal{H}(\bm{x}_{\nicefrac{(j-1)}{2s}}(\tau_{j-1})) =    \mathcal{D}_h + \mathcal{S}_h \leq  \mathcal{S}_h \\
        \text{with} \quad \mathcal{D}_h & = \sum_{j=1}^{2s} \int_{\tau_{j-2}}^{\tau_j} - \bm{z}( \bm{x}_{\nicefrac{j}{2s}}(\tau) )^\top \bm{R}_{i(j)}( \bm{x}_{\nicefrac{j}{2s}}(\tau)  ) \bm{z}( \bm{x}_{\nicefrac{j}{2s}}(\tau) ) \, \mathrm{d}\tau \leq 0,\\
 \mathcal{S}_h &=  \sum_{j=1}^{2s} \int_{\tau_{j-2}}^{\tau_j}   \bm{y}_{i(j)}(\tau)^\top \bm{u}(\tau) \, \mathrm{d}\tau =  \sum_{j=1}^{2s} \int_{\tau_{j-2}}^{\tau_j} \bm{z}(\bm{x}_{\nicefrac{j}{2s}}(\tau))^\top \bm{B}_{i(j)}(\bm{x}_{\nicefrac{j}{2s}}(\tau)) \bm{u}(\tau) \, \diff \tau . \nonumber
        \end{align}
For $\bm B=\bm 0$, $\bm B_1=\bm B_2=\bm 0$ holds by assumption, hence we find $\mathcal{S}_h=0$. For $\bm R=\bm 0$, we have $\bm R_1=\bm R_2=\bm 0$ because of $\bm{R}_i(\bm{x})\succeq \bm{0}$ for all $\bm x$. Thus, $\mathcal{D}_h=0$ holds.

It remains to show that $\mathcal{D}_h$ and $\mathcal{S}_h$ are consistent approximations of the continuous dissipated energy and supplied energy, respectively.    
Taylor expansions with the imposed regularity assumptions yield  $\bm{x}_{\nicefrac{j}{2s}}(t) = \bm{x}_0 + \mathcal{O}(h)$, $\bm{y}_{i}(t) = \bm{y}_i(t_0) +  \mathcal{O}(h)$, and $\bm{u}(t) = \bm{u}(t_0) +  \mathcal{O}(h)$ for all $t \in [\min_j \tau_j, \max_j \tau_j]$. Hence, we get
  \begin{align*}
  \mathcal{D}_h
  &= -h\Big(\textstyle\sum_j a_j\Big)\, (\bm{z}^\top \bm{R}_1\, \bm{z}) \big|_{\bm{x}_0}
    -h\Big(\textstyle\sum_j b_j\Big)\, (\bm{z}^\top \bm{R}_2\, \bm{z} )\big|_{\bm{x}_0}
    + \mathcal{O}(h^2)
  = -h\, (\bm{z}^\top \bm{R}\, \bm{z} )\big|_{\bm{x}_0} + \mathcal{O}(h^2), \\
  \mathcal{S}_h
  &= \phantom{-}h\Big(\textstyle\sum_j a_j\Big)\, (\bm{y}_1^{\top} \bm{u} )\big|_{t_0}
    +h\Big(\textstyle\sum_j b_j\Big)\, (\bm{y}_2^{\top} \bm{u}) \big|_{t_0}
    + \mathcal{O}(h^2)
  = h\, (\bm{y}^\top \bm{u}) \big|_{t_0} + \mathcal{O}(h^2),
 \end{align*}
using the consistency condition $\sum_j a_j = \sum_j b_j = 1$ together with $\bm{R}_1+\bm{R}_2 = \bm{R}$ and $(\bm{y}_1 + \bm{y}_2)(t_0) = ((\bm{B}_1+\bm{B}_2)(\bm x_0))^\top \bm{z}(\bm x_0) = \bm{y}(t_0)$. Dividing by $h$ and letting $h\to 0$ yields the desired result
$$
  \lim_{h\to0}\frac{\mathcal{D}_h}{h} = -\bm{z}(\bm{x}_0)^\top \bm{R}(\bm{x}_0)\,\bm{z}(\bm{x}_0),
  \qquad
  \lim_{h\to0}\frac{\mathcal{S}_h}{h} = \bm{y}(t_0)^\top \bm{u}(t_0).
    $$
\end{proof}

\begin{proposition}\label{prop: SplittingOrder2}
Let the hypotheses of Lemma \ref{lemma: energy-consistent splitting} and the following regularity assumptions hold: Hamiltonian $\mathcal{H}\in \mathcal{C}^{r+1}(\mathbb{R}^n,\mathbb{R})$, flow matrix function $\bm E \in \mathcal{C}^{r}(\mathbb{R}^n,\mathbb{R}^{n\times n})$, port functions $\bm B, \bm B_i \in \mathcal{C}^r(\mathbb{R}^n,\mathbb{R}^{n\times m})$, input $\bm u \in \mathcal{C}^r([t_0,T],\mathbb{R}^m)$ and $\bm{f},\bm f_i:[t_0,T]\times \mathbb{R}^n\rightarrow \mathbb{R}^n$ Lipschitz continuous in $\bm x$ for $r=1$ and continuously differentiable for $r=2$.
Let the splitting scheme $\bm{\Psi}$ be of consistency order $p\leq 2$ and $r=p$, then it is energy-consistent of order $q=p$.
\end{proposition}

\begin{proof}
According to Lemma \ref{lemma: energy-consistent splitting}, any energy-consistent splitting method $\bm{\Psi}$ satisfies \eqref{eq:splitting-method_power-balance}.  
The intermediate solution $\bm{x}_{\nicefrac{j}{2s}}$ can be expressed by the exact subflows as
\begin{align*}
    \bm{x}_{\nicefrac{j}{2s}}(\tau) &= \bm{x}_0 + \sum_{k=1}^{j-1} \int_{\tau_{k-2}}^{\tau_{k}}   \bm E^{-1}(\bm x(s)) \bm{f}^{[i(k)]}(s,\bm x(s)) \, \mathrm{d}s + \int_{\tau_{j-2}}^{\tau}  \bm E^{-1}(\bm x(s)) \bm{f}^{[i(j)]}(s,\bm x(s))  \, \mathrm{d}s 
\end{align*}    
for $\tau$ in the $j$-th sub-step from $\tau_{j-2}$ to $\tau_j$. Consider the case $p=2$. We approximate the integrals with a left-endpoint rectangular quadrature rule and insert the intermediate solution formula, recursively. Local expansions around $\bm x_0$ then yield 
\begin{align*}
    \bm{x}_{\nicefrac{j}{2s}}(\tau) &= \bm{x}_0 + \bm{v}_j(\tau;t_0,\bm{x}_0) + \mathcal{O}(h^2), \\
    \bm{v}_j(\tau;t_0,\bm{x}_0) &= h   \bm E^{-1}(\bm x_0) [  A_{\lceil \nicefrac{(j-1)}{2} \rceil} \bm{f}^{[1]}(t_0,\bm{x}_0) +  B_{\lfloor \nicefrac{(j-1)}{2} \rfloor} \bm{f}^{[2]}(t_0, \bm{x}_0)] \\
    &  \quad + (\tau - \tau_{j-2}) \bm E^{-1}(\bm x_0)\bm{f}^{[i(j)]}(t_0,\bm{x}_0)
\end{align*}
with Gaussian brackets $\lceil \cdot \rceil$, $\lfloor \cdot \rfloor$, and $|\tau - \tau_{j-2}|<h$.
Inserting these expressions into the supply terms of $\mathcal{S}_h$ in \eqref{eq:splitting-method_power-balance} and linearizing around $h=0$ gives
\begin{align*}
    \bm{z}(\bm{x}_{\nicefrac{j}{2s}}(\tau))^\top \bm{B}_{i(j)}(\bm{x}_{\nicefrac{j}{2s}}(\tau)) \bm{u}(\tau) =\; \bm{z}(\bm{x}_0)^\top \bm{B}_{i(j)}(\bm{x}_0) \bm{u}(t_0)  + (\tau - t_0)\, \bm{z}(\bm{x}_0)^\top \bm{B}_{i(j)}(\bm{x}_0) \dot{\bm{u}}(t_0) \\     
+  \bm{v}_j(\tau;t_0,\bm{x}_0)^\top \mathrm{D}\bm{z}(\bm{x}_0)^\top \bm{B}_{i(j)}(\bm{x}_0) \bm{u}(t_0)   + \bm{z}(\bm{x}_0)^\top \mathrm{D}\bm{B}_{i(j)}(\bm{x}_0)[\bm{v}_j(\tau;t_0,\bm{x}_0)] \bm{u}(t_0) + \mathcal{O}(h^2).
\end{align*}
Integration using the midpoint rule identity leads to
\begin{align*}
    &\int_{\tau_{j-2}}^{\tau_j} \bm{z}(\bm{x}_{\nicefrac{j}{2s}}(\tau))^\top \bm{B}_{i(j)}(\bm{x}_{\nicefrac{j}{2s}}(\tau)) \bm{u}(\tau) \, \diff \tau \\
    & =(\tau_j - \tau_{j-2}) 
    \bigg[ \bm{z}(\bm{x}_0)^\top \bm{B}_{i(j)}(\bm{x}_0) \bm{u}(t_0) + \Big( \tfrac{\tau_j + \tau_{j-2}}{2} - t_0 \Big) \bm{z}(\bm{x}_0)^\top \bm{B}_{i(j)}(\bm{x}_0) \dot{\bm{u}}(t_0) \\
    & \quad +\bm{v}_j\big(\tfrac{\tau_j + \tau_{j-2}}{2};t_0,\bm{x}_0\big)^\top \mathrm{D}\bm{z}(\bm{x}_0)^\top \bm{B}_{i(j)}(\bm{x}_0) \bm{u}(t_0)
    +   \bm{z}(\bm{x}_0)^\top \mathrm{D}\bm{B}_{i(j)}(\bm{x}_0)\big[ \bm{v}_j\big(\tfrac{\tau_j + \tau_{j-2}}{2};t_0,\bm{x}_0\big) \big] \bm{u}(t_0)  \bigg] \\
    &\quad + \mathcal{O}(h^3).
\end{align*}
Algebraic manipulations of the second-order consistency conditions arising from the Baker--Campbell--Hausdorff formula imply
\begin{align*}
    \sum\nolimits_{k=1}^s \big(a_k A_{k-1} + \tfrac{a_k^2}{2} \big) = \sum\nolimits_{k=1}^s \big( b_k B_{k-1} + \tfrac{b_k^2}{2} \big) = \sum\nolimits_{k=1}^s a_k B_{k-1} = \sum\nolimits_{k=1}^s b_k A_k = \tfrac{1}{2} ,
\end{align*}
such that summing up all expansions of the supplied energy contributions results in
\begin{align*}
    \mathcal{S}_h &=  h \bm{z}(\bm{x}_0)^\top \bm{B}(\bm{x}_0) \bm{u}(t_0) + \tfrac{h^2}{2} \Big[ \bm{z}(\bm{x}_0)^\top \bm{B}(\bm{x}_0) \dot{\bm{u}}(t_0) \\
    &\quad + (\bm E^{-1}(\bm{x}_0)\bm{f}(t_0,\bm{x}_0))^\top \mathrm{D}\bm{z}(\bm{x}_0)^\top \bm{B}(\bm{x}_0) \bm{u}(t_0) 
   + \bm{z}(\bm{x}_0)^\top \mathrm{D}\bm{B}(\bm{x}_0)  [\bm E^{-1}(\bm{x}_0)\bm{f}(t_0,\bm{x}_0)] \bm{u}(t_0) \Big] + \mathcal{O}(h^3)
\end{align*}

Consider now the exact solution $\bm x(t)=\bm{\varphi}_{t,t_0}(\bm x_0)$, $t\in [t_0,t_0+h]$. Expanding  the exact supplied energy  $\mathcal{S}^\star=\int_{t_0}^{t_0 + h} \bm{y}(\tau)^\top \bm{u}(\tau) \, \diff \tau$ with $\bm y(t)=\bm B(\bm x(t))^\top \bm z(\bm x(t))$ in $h$,  we find
  $\mathcal{S}_h = \mathcal{S}^\star + \mathcal{O}(h^3)$.
Moreover, the second-order consistency of the scheme and the regularity of $\mathcal{H}$ imply
$$ \mathcal{H}(\bm{x}_1) - \mathcal{H}(\bm{x}_0) = \mathcal{H}(\bm{x}(t_0+h)) - \mathcal{H}(\bm{x}_0) + \mathcal{O}(h^3)= \mathcal{D}^\star + \mathcal{S}^\star + \mathcal{O}(h^3)$$
with exact dissipated energy $\mathcal{D}^\star$.
Thus, we obtain 
$$ \mathcal{D}_h = (\mathcal{D}_h + \mathcal{S}_h) - \mathcal{S}_h = \big(\mathcal{H}(\bm{x}_1) - \mathcal{H}(\bm{x}_0)\big) - \mathcal{S}_h = \mathcal{D}^\star + \mathcal{O}(h^3),$$
concluding the proof. The case $p=1$ is covered. 
\end{proof}

At first glance, it might seem that assuming sufficient regularity of the pH-ODE functions, the order of consistency $p$ always carries over to the energy consistency order $q$. But this may differ in the case for higher-order splitting. Classical higher-order schemes ($p\geq 3$) involve negative step size coefficients. These schemes can only be energy-consistent for pH-ODEs without dissipation ($\bm R=\bm 0$) according to Lemma~\ref{lemma: energy-consistent splitting}, then $\mathcal{D}_h=0$, and indeed $q=p$ under certain regularity assumptions. Higher-order energy-consistent splitting schemes for pH-ODEs with dissipation ($\bm R\neq \bm 0$) are designed on the basis of commutators and/or special decompositions \cite{moench2025}. These schemes are not of higher order for general vector fields $\bm f$. The step size coefficients are positive and sum up to one for each subproblem (consistency condition), but they do not fulfill all other order conditions from the Baker--Campbell--Hausdorff formula, instead some conditions vanish due to the introduced commutators or the Lie derivatives of the considered subflows. In this sense, the step size coefficients induce a quadrature rule, but not necessarily of same order, i.e. $\mathcal{S}_h=\mathcal{S}^\star + \mathcal{O}(h^{q+1})$, $q\leq p$, cf.\ Example~\ref{ex:order}.

\begin{example}\label{ex:order}
Consider a pH-ODE \eqref{eq:pH-ODE} with $\mathcal{H}(\bm x)=\tfrac{1}{2} \bm x^\top \bm Q \bm x$, $\bm E=\bm I$ and constant system matrices $\bm J$, $\bm R$, $\bm B$.
The port-based splitting approach from \cite{moench2025} is based on a decomposition where the first subproblem contains the ports ($\splitpart{1}:$  $\dot{\bm x}=\bm B \bm u(t)$, $\dot t=0$) and the second subproblem the inner dynamics ($\splitpart{2}:$  $\dot{\bm x}=(\bm J-\bm R) \bm Q \bm x$, $\dot t=1$). This decomposition causes the vanishing of several higher-order commutators such that fourth order can be achieved with symmetric 3-stage schemes and sixth order with symmetric 4-stage schemes. But it only induces a frozen-time quadrature-type approximation for the discrete power balance, i.e.,
\begin{align*}
 \mathcal{H}(\bm{x}_1) - \mathcal{H}(\bm{x}_0)\leq \mathcal{S}_h= \sum_{k=1}^{s} \int_{\tau_{(2k-1)-2}}^{\tau_{2k-1}}   \bm{y}_{i(2k-1)}(\tau)^\top \mathrm{d}\tau \, \bm{u}(\tau_{2k-2}). 
 \end{align*}
 
Figure~\ref{fig:energy-consistency-pbs} illustrates the approximation quality of state, Hamiltonian, dissipated energy and supplied energy for the port-based splitting schemes \texttt{PBS4} and \texttt{PBS6} applied to the damped and driven harmonic oscillator,
  \begin{equation*}
        \dot{\bm{x}} = \left[\begin{pmatrix}
            0 & -1 \\ 1 & 0
        \end{pmatrix} - \begin{pmatrix}
            d & 0 \\ 0 & 0
        \end{pmatrix}\right] \begin{pmatrix}
            1 & 0 \\ 0 & k
        \end{pmatrix} \bm{x} + \begin{pmatrix}
            -1 \\ 0
        \end{pmatrix} u(t), \qquad \bm{x}(0)=\begin{pmatrix} 1\\0 \end{pmatrix},
    \end{equation*}
    with $d=1$, $k=1000$, and $u(t) = 5\cos(3 t)$, $t\in [0,T]$, $T=1$. In particular, the errors at $t=T$ are shown in dependence on the step size $h$, i.e., $\|\bm x_{T,h}-x(T)\|_2$, $|\mathcal{H}(\bm x_{T,h})-\mathcal{H}(x(T))|$, $|\mathcal{D}_{T,h}-\mathcal{D}^\star|$, and  $|\mathcal{S}_{T,h}-\mathcal{S}^\star|$, where $\bm x_{T,h}$ refers to the approximated state at $t=T$ computed with step size $h$ and $\bm x(T)$ is the exact solution with the associated exact energies $\mathcal{D}^\star$, $\mathcal{S}^\star$. State and Hamiltonian converge with order $p=4$ for \texttt{PBS4} and $p=6$ for \texttt{PBS6} in accordance to the designed consistency orders, whereas the dissipated energy and the supplied energy exhibit only second order, indicating an energy consistency order $q = 2$.
\end{example}

\begin{figure}[t]
    \centering
     \includegraphics[]{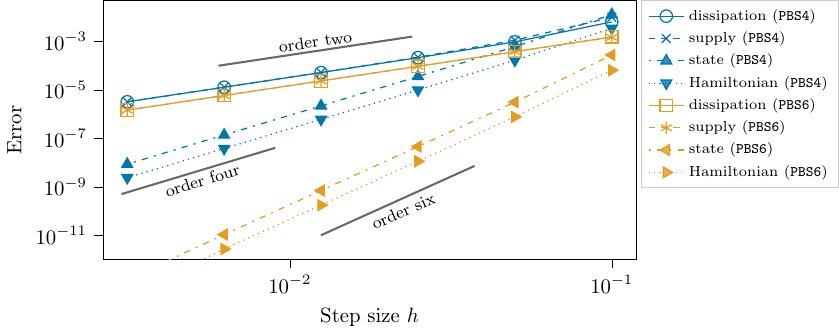}
        \caption{Approximation errors in the state, Hamiltonian, dissipated energy, and supplied energy of the port-based splitting schemes \texttt{PBS4} ($p=4$) and \texttt{PBS6} ($p=6$) for different step sizes $h$, compared with the exact solution of a damped and driven harmonic oscillator.}
    \label{fig:energy-consistency-pbs}
\end{figure}

\subsection{Subflow Approximation}

For numerical simulation it is convenient and often necessary to replace the exact subflows by numerical approximations. The resulting method $\bm{\Psi}^\mathrm{A}$ retains the consistency order $p$ from the underlying splitting scheme  $\bm{\Psi}$ provided that the subflow approximations are of the same or higher order. As discussed, under the stated assumptions, the orders of consistency and energy consistency of a numerical integrator satisfy $q\leq p$.

\begin{proposition}\label{prop: Energyconsistentsplit}
Let the hypotheses of Lemma \ref{lemma: energy-consistent splitting} and the regularity assumptions of Proposition~\ref{prop: SplittingOrder2} hold.
Let $\bm{\Psi}$ be a splitting method that is energy-consistent of order $q\leq2$. If each subflow $\bm{\varphi}^{[i]}$ is replaced by a numerical flow $\bm{\psi}^{[i]}$ that is energy-consistent to the same or a higher order, then the overall method $\bm{\Psi}^\mathrm{A}$ is energy-consistent of order $q^\mathrm{A}=q$.
\end{proposition}

\begin{proof}
In the splitting scheme $\bm{\Psi}$, dissipated energy $\mathcal{D}_h$ and supplied energy $\mathcal{S}_h$ are given by \eqref{eq:splitting-method_power-balance}.
We denote the energy contributions in the $j$-th sub-step, $j=1,...,2s$, by
$$\mathcal{D}^j  = \int_{\tau_{j-2}}^{\tau_j} \hspace*{-0.3cm} - \bm{z}( \bm{x}_{\nicefrac{j}{2s}}(\tau) )^\top \bm{R}_{i(j)}( \bm{x}_{\nicefrac{j}{2s}}(\tau)  ) \bm{z}( \bm{x}_{\nicefrac{j}{2s}}(\tau) )\leq 0, \qquad
 \mathcal{S}^j = \int_{\tau_{j-2}}^{\tau_j}   \bm{y}_{i(j)}(\tau)^\top \bm{u}(\tau) \, \mathrm{d}\tau, $$
they depend on the intermediate solution $\bm x_{\nicefrac{j}{2s}}(\tau)= \bm{\varphi}^{[i(j)]}_{\tau,\tau_{j-2}}(\bm x_{\nicefrac{(j-1)}{2s}}(\tau_{j-1}))$ with exact subflow $\bm \varphi^{[i]}$.
Since $\bm{\Psi}^\mathrm{A}$ is composed of energy-consistent numerical integrators (subflow approximations) $\bm \psi^{[i]}$, a discrete power balance is fulfilled in each sub-step. The corresponding energy contributions, denoted by $\hat {\mathcal{D}}^j\leq 0$ and $\hat{\mathcal{S}}^j$, depend on the associated intermediate (numerical) solution $\hat{\bm x}_{\nicefrac{j}{2s}}(\tau)= \bm{\psi}^{[i(j)]}_{\tau,\tau_{j-2}}(\hat{\bm x}_{\nicefrac{(j-1)}{2s}}(\tau_{j-1}))$. 
Note that
\begin{align}\label{eq:int-sol}
\|\hat{\bm x}_{\nicefrac{j}{2s}}-\bm x_{\nicefrac{j}{2s}}\|=\mathcal{O}(h^{q+1}) \text{ for all } j,
\end{align}
which can be concluded by induction. For $j=1$, due to the exact initialization with $\bm x_0$ at $t_0$, we find directly the consistency error of $\bm{\psi}^{[i]}$, whose order satisfies $p^{[i]} \geq q$. For $j>1$, $i=i(j)$, we have
\begin{align*}
\big\|\hat{\bm x}_{\nicefrac{j}{2s}}-\bm x_{\nicefrac{j}{2s}}\big\|
&\leq \big\|(\bm{\psi}^{[i]}_{.,\tau_{j-2}}-\bm{\varphi}^{[i]}_{.,\tau_{j-2}})\big|_{\hat{\bm x}_{\nicefrac{(j-1)}{2s}}(\tilde \tau)}\big\|+
 \big\|\bm{\varphi}^{[i]}_{.,\tau_{j-2}}({\hat{\bm x}_{\nicefrac{(j-1)}{2s}}(\tilde \tau)})-\bm{\varphi}^{[i]}_{.,\tau_{j-2}}({\bm x_{\nicefrac{(j-1)}{2s}}(\tilde\tau)})\big\|,
\end{align*}
$\tilde \tau=\tau_{j-1}$, hence the estimate \eqref{eq:int-sol} follows from the consistency of $\bm{\psi}^{[i]}$ and the Lipschitz continuity of $\bm{\varphi}^{[i]}$. 

Introducing $\tilde{\bm x}_{\nicefrac{j}{2s}}(\tau)= \bm{\varphi}^{[i(j)]}_{\tau,\tau_{j-2}}(\hat{\bm x}_{\nicefrac{(j-1)}{2s}}(\tau_{j-1}))$, the respective energies $\mathcal{D}^j_{\tilde{\bm x}}$, $\mathcal{S}^j_{\tilde{\bm x}}$ are the exact counterparts to $\hat{\mathcal{D}}^j$, $\hat{\mathcal{S}}^j$, i.e., $|\hat{\mathcal{D}}^j-\mathcal{D}^j_{\tilde{\bm x}}|=\mathcal{O}(h^{q+1})$ and $|\hat{\mathcal{S}}^j-\mathcal{S}^j_{\tilde{\bm x}}|=\mathcal{O}(h^{q+1})$ by the energy consistency of $\bm{\psi}^{[i]}$. Moreover
\begin{align*}
|\mathcal{D}^j_{\tilde{\bm x}} - \mathcal{D}^j|
&\leq \int_{\tau_{j-2}}^{\tau_{j}} \|(\bm z^\top \bm R_{i(j)} \bm z)\circ \bm{\varphi}^{[i(j)]}_{\cdot,\tau_{j-2}}({\hat{\bm x}_{\nicefrac{(j-1)}{2s}}(\tilde \tau)})- (\bm z^\top \bm R_{i(j)} \bm z)\circ \bm{\varphi}^{[i(j)]}_{\cdot,\tau_{j-2}}({{\bm x}_{\nicefrac{(j-1)}{2s}}(\tilde \tau)})\| \, \mathrm{d} \tau \\
& \leq (\tau_j-\tau_{j-2}) L^R_j \|\hat{\bm x}_{\nicefrac{(j-1)}{2s}}-\bm x_{\nicefrac{(j-1)}{2s}}\| =\mathcal{O}(h^{q+1})
\end{align*}
with Lipschitz constant $L^R_j$ of  $(\bm z^\top \bm R_{i(j)} \bm z)\circ \bm{\varphi}^{[i(j)]}$ and \eqref{eq:int-sol}. Analogously,
\begin{align*}
|\mathcal{S}^j_{\tilde{\bm x}} - \mathcal{S}^j|
& \leq |\tau_j-\tau_{j-2}| \|\bm u\|_\infty L^B_j \|\hat{\bm x}_{\nicefrac{(j-1)}{2s}}-\bm x_{\nicefrac{(j-1)}{2s}}\|  =\mathcal{O}(h^{q+1})
\end{align*}
with Lipschitz constant $L^B_j$ of  $(\bm z^\top \bm B_{i(j)})\circ \bm{\varphi}^{[i(j)]}$. Applying the triangle inequality we hence have for the energy distributions in the $j$-th sub-step
\begin{align*}
|\hat{\mathcal{D}}^j-\mathcal{D}^j|=\mathcal{O}(h^{q+1}), \qquad |\hat{\mathcal{S}}^j-\mathcal{S}^j|=\mathcal{O}(h^{q+1}) \text{ for all } j
\end{align*}
Summing up over all substeps yields the result.
 \end{proof}
 
Note that the result from Proposition~\ref{prop: Energyconsistentsplit} can be transferred straightforwardly to higher order ($q\geq 3$).
As for energy-consistent numerical integrators $\bm{\psi}^{[i]}$, Gauss collocation schemes can be used for quadratic Hamiltonians and discrete gradient methods for general Hamiltonians. Gauss collocation schemes satisfy $q=p$ and are available for arbitrary order $p=2s$ with $s$ stages \cite{hairer2006}. Discrete gradient schemes of consistency order $p\leq 2$ are also well established \cite{gonzalez1996, kinon2026}, in particular $q=p$ holds (cf.\ Appendix~\ref{app:DGM}). Higher-order energy-consistent discrete gradient schemes are limited to systems without dissipation $\bm R = \bm 0$ (analogously as for classical splitting), \cite{celledoni2017, eidnes2022}.

\section{Decomposition Strategies} \label{sec:decomposition}

The decomposition of a port-Hamiltonian system has to balance two, in general competing, objectives:  preserving the energetic properties of the continuous system and exploiting the structure of the underlying model to reduce the computational cost. In particular, for coupled systems, the interconnection of the subsystems induces a specific block form that offers potential for dimension reduction and parallelization.
Consider a pH-ODE $\system$ \eqref{eq:pH-ODE} obtained by an energy-conserving skew-symmetric coupling of two pH-ODE subsystems $\system_1$ and $\system_2$ of dimensions $n_1$ and $n_2$, respectively. With $\bm x=(\bm x_1^\top, \bm x_2^\top)^\top$ and $n=n_1+n_2$, it is given by \cite{ehrhardt2026}
\begin{equation} \label{eq:coupledSystem}
\begin{aligned}
\begin{pmatrix}
\bm{E}_1(\bm{x}_1) & \bm{0}\\
\bm{0} & \bm{E}_2(\bm{x}_2)
\end{pmatrix}\dot{\bm{x}}
&=
\bigl(\begin{pmatrix}
\bm{J}_1(\bm{x}_1) & \bm{C}(\bm{x})\\[0.3em]
-\bm{C}(\bm{x})^{\top} & \bm{J}_2(\bm{x}_2)
\end{pmatrix}
-
\begin{pmatrix}
\bm{R}_1(\bm{x}_1) & \bm{0}\\
\bm{0} & \bm{R}_2(\bm{x}_2)
\end{pmatrix}
\bigr)
\begin{pmatrix}
\bm{z}_1(\bm{x}_1)\\
\bm{z}_2(\bm{x}_2)
\end{pmatrix}\\
&\quad +
\begin{pmatrix}
\bm{B}_1(\bm{x}_1) & \bm{0}\\
\bm{0} & \bm{B}_2(\bm{x}_2)
\end{pmatrix}
\begin{pmatrix}
\bm{u}_1(t)\\
\bm{u}_2(t)
\end{pmatrix},\\
(\bm{y}_1^\top,\bm{y}_2^\top)^\top&=\bm B(\bm x)^\top \bm z (\bm x)
\end{aligned}
\end{equation}
with coupling matrix $\bm{C}(\bm{x})$. The Hamiltonian is additive, $\mathcal{H}(\bm x)= \mathcal{H}_1(\bm x_1)+ \mathcal{H}_2(\bm x_2)$.

In this section we present five decomposition strategies. We first describe general energy- and port-based decompositions and their specialization to a coupled system~\eqref{eq:coupledSystem}, followed by decompositions that explicitly exploit the coupling structure. Finally, we discuss their hierarchical combination to address the complexity of multiphysical problems.
In the following, we call a decomposition \emph{energy-consistent} if it yields an energy-consistent splitting scheme under the assumptions of Lemma~\ref{lemma: energy-consistent splitting}.

\subsection{Energy- and Port-based Decompositions}\label{sec:structure-preserving}

The energy-associated and port-based decompositions provide general mechanisms for splitting a port-Hamiltonian system while retaining the energetic properties of the underlying formulation; they are therefore energy-consistent. They apply to general port-Hamiltonian systems \eqref{eq:pH-ODE} and are thus not restricted to coupled systems \eqref{eq:coupledSystem}. For a coupled system, however, the block structure provides additional computational advantages.

\subsubsection*{\textbf{Energy-associated decomposition (conservative vs.\ passive)}}

The energy-associated decomposition separates the energy-conserving interconnection dynamics from dissipation and external forcing. Applied to $\system$, it gives the conservative subproblem $\splitpart{1}$ and the passive subproblem $\splitpart{2}$ \cite{frommer2026,moench2025},
\begin{subequations}\label{eq:JR-split}
\begin{align}
\label{eq:JR-split-conservative}
\splitpart{1} \colon \,\,
\bm{E}(\bm{x})\dot{\bm{x}} &= \bm{J}(\bm{x})\,\bm{z}(\bm{x}) && \text{(conservative),} \\
\label{eq:JR-split-dissipative}
\splitpart{2} \colon \,\,
\bm{E}(\bm{x})\dot{\bm{x}} &=
-\bm{R}(\bm{x})\,\bm{z}(\bm{x}) + \bm{B}(\bm{x})\bm{u}(t) && \text{(passive).}
\end{align}
\end{subequations}
Obviously, the decomposition satisfies the hypotheses of Lemma~\ref{lemma: energy-consistent splitting} and thus enables the construction of energy-consistent splitting methods.

For the coupled system~\eqref{eq:coupledSystem}, the block-diagonal structure of $\bm{E}$, $\bm{R}$, and $\bm B$ implies that $\splitpart{2}$ separates into two systems of dimensions $n_1$ and $n_2$, i.e.,
\[
\bm E_i(\bm x_i)\dot{\bm x}_i=-\bm R_i(\bm x_i)\bm z_i(\bm x_i)+\bm B_i(\bm x_i)\bm u_i(t), \qquad i=1,2,
\]
which can thus be solved independently and, if appropriate, in parallel. In contrast, $\splitpart{1}$ contains the complete interconnection matrix $\bm J(\bm x)$ with coupling matrix $\bm C(\bm x)$ and has hence full dimension $n$ even if the coupling is low-dimensional.

The decomposition separates distinct energetic roles and hence permits different numerical treatments of the two subproblems. While the passive subproblem can be treated by any dissipative integrator, the conservative subproblem should be approximated by an energy-conserving method \cite{hairer2006}. For quadratic Hamiltonians, Gauss collocation and symplectic Runge--Kutta methods conserve the Hamiltonian exactly; for general Hamiltonians, discrete gradient methods provide a corresponding energy-consistent approximation. 

Note that the energy-associated decomposition forms the basis for higher-order energy-con\-sistent splitting methods  ($p\geq3$) for linear pH-ODEs \cite{moench2025,schaefers2026}, and certain nonlinear subclasses \cite{moench2026}. Under certain assignments of the energy parts in the constraints, it can be also used for splitting of index-1 pH-DAEs \cite{bartel2025}.

\subsubsection*{\textbf{Port-based decomposition (internal vs.\ external)}}

The port-based decomposition separates the internal dynamics from the interaction with the external environment. It is given by \cite{auzinger2019, moench2025}
\begin{subequations}\label{eq:port-split}
\begin{align}
\label{eq:port-split-int}
\splitpart{1}
\colon \,\,
\bm{E}(\bm {x})\dot{\bm{x}} &= \left( \bm{J}(\bm{x}) - \bm{R}(\bm{x})\right)\,\bm{z}(\bm{x}) && \text{(internal),} 
\\
\label{eq:port-split-ext}
\splitpart{2}
\colon \,\,
\bm{E}(\bm {x})\dot{\bm{x}} &= \bm{B}(\bm{x})\,\bm{u}(t) && \text{(external).}
\end{align}
\end{subequations}
The two subproblems have a direct energetic interpretation: the first accounts for internal energy exchange and dissipation, and the second represents energy exchange through the external ports.
While the internal subproblem $\splitpart{1}$ is a closed, autonomous dissipative pH-ODE, the external subproblem $\splitpart{2}$ is lossless, satisfying $\tfrac{\mathrm d}{\mathrm dt}\mathcal{H}(\bm x(t))=\bm y(t)^\top \bm u(t)$.
The decomposition fulfills the hypotheses of Lemma~\ref{lemma: energy-consistent splitting}, allowing for energy-consistent splitting.

For the coupled system~\eqref{eq:coupledSystem}, the block-diagonal structure of $\bm{E}$ and $\bm{B}$ causes $\splitpart{2}$ to separate into 
\[
\bm E_i(\bm x_i)\dot{\bm x}_i=
\bm B_i(\bm x_i)\bm u_i(t), \quad i=1,2.
\]
Thus, the port contribution can be evaluated independently -- in parallel -- for the resulting subsystems of size $n_1$ and $n_2$. In contrast,  $\splitpart{1}$ contains the coupling through $\bm{C}(x)$ and generally remains a coupled system, as in the energy-associated decomposition.

For linear systems, the port-based decomposition can further be exploited in the construction of higher-order splitting methods. Problem-specific commutator relations, together with frozen-time techniques in the non-autonomous case \cite{blanes2006, blanes2010}, may cause higher-order commutators to vanish and thereby permit higher-order compositions with positive coefficients. The resulting methods remain compatible with the energy consistency framework, although the order of approximation of the supplied and dissipated energies need not coincide with the classical order of the splitting method (cf.\ Example~\ref{ex:order}).

\subsection{Decompositions Exploiting the Coupling Structure}\label{sec:coupling-decompositions}

For coupled pH-ODEs~\eqref{eq:coupledSystem}, the interconnection structure provides additional opportunities to enhance computational efficiency. The sub\-system-based decomposition follows directly from the physical subsystem partition and yields the most direct reduction of the effective problem dimensions. The diagonal decomposition instead separates the uncoupled subsystem dynamics from the coupling, enabling dimensional reduction and parallel computation, while preserving the port-Hamiltonian structure of both subproblems. The time-scale decomposition exploits separated time scales and allows fast and slow dynamics to be resolved on different time grids.

\subsubsection*{\textbf{Subsystem-based decomposition (component-wise)}}

Network-based modeling is a natural approach for large multiphysical systems, in which a complex model is assembled from smaller subsystems. From a computational perspective, it is therefore appealing to reverse this construction in the numerical treatment and solve the individual subsystems separately. This idea leads to the subsystem-based decomposition (cf.\ component-wise partitioning in \cite{arnold2001,busch2012,kuebler2000}). A row-wise decomposition of the system matrices yields two subproblems associated with the two underlying (physical) subsystems $\subsystem{1}$ and  $\subsystem{2}$,
\begin{subequations}\label{eq:dim-split}
\begin{align}
\label{eq:dim-split-1}
\splitpart{1} \colon \,\, \bm{E}(\bm{x})\dot{\bm{x}}  &=
\Biggl[
\underbrace{\begin{pmatrix}
        \bm{J}_1(\bm{x}_1) & \bm{C}(\bm{x})\\
        \bm{0} & \bm{0}
    \end{pmatrix}}_{=:\, \bm{J}^{[1]}(\bm{x})}
-
\begin{pmatrix}
        \bm{R}_1(\bm{x}_1) & \bm{0}\\
        \bm{0} & \bm{0}
    \end{pmatrix}
\Biggr] \bm{z}(\bm{x})
+
\begin{pmatrix}
        \bm{B}_{1}(\bm{x}_1)\bm{u}_1(t)\\
        \bm{0}
    \end{pmatrix} &&  
        (\subsystem{1}),
        \\
\label{eq:dim-split-2}
\splitpart{2} \colon \,\, \bm{E}(\bm{x})\dot{\bm{x}}  &=
\Biggl[
\underbrace{\begin{pmatrix}
        \bm{0} & \bm{0}\\
        -\bm{C}(\bm{x})^{\top} & \bm{J}_2(\bm{x}_2)
    \end{pmatrix}}_{=:\, \bm{J}^{[2]}(\bm{x})}
-
\begin{pmatrix}
        \bm{0} & \bm{0}\\
        \bm{0} & \bm{R}_2(\bm{x}_2)
    \end{pmatrix}
\Biggr] \bm{z}(\bm{x})
+
\begin{pmatrix}
        \bm{0}\\
        \bm{B}_{2}(\bm{x}_2)\bm{u}_2(t)
    \end{pmatrix} &&     
     (\subsystem{2}).
\end{align}
\end{subequations}
The characteristic feature of this decomposition is the inactive block row: $\bm E_2(\bm x_2)\dot{\bm x}_2=\bm 0$ in $\splitpart{1}$ and $\bm E_1(\bm x_1)\dot{\bm x}_1=\bm 0$ in $\splitpart{2}$. Although the subproblems are formally posed in the full state space of dimension $n$, their dynamics therefore evolve effectively in dimensions $n_1$ and $n_2$, respectively. This can substantially reduce the computational cost and allows the two subproblems to be treated independently.

The computational advantage, however, comes at a structural cost. Unless $\bm C(\bm x)=\bm0$, the matrices $\bm J^{[i]}(\bm x)$ are not skew-symmetric, and the effort/flow pairing is disrupted, \cite{polyuga2012, vanderschaft2014}. Note that $\bm C(\bm x)=\bm0$ corresponds to the absence of any coupling between $\subsystem{1}$ and $\subsystem{2}$ and is therefore not relevant to the coupled systems considered here. Consequently, the subproblems are, in general, not port-Hamiltonian systems with respect to the original gradient pair $(\bm E,\bm z)$. Their flows therefore need not satisfy the power balance of the original system.

This loss of structure has direct consequences for the numerical behavior of the resulting splitting method. Even in the absence of external input, the Hamiltonian may increase along an individual subflow, so that the dissipation inequality of the original system is not inherited by the splitting. Hence, the dimensional reduction achieved by the subsystem-based decomposition does not by itself imply energy consistency.

For linear systems, the implications can be made more explicit. The stability of an individual subsystem-based subproblem depends on the spectrum of its reduced system matrix and is not automatically inherited from the stability of the original coupled system. Even if the individual subproblems are stable, their composition, for example via Strang splitting, may require a step-size restriction, as the spectral radius of the resulting splitting matrix can exceed one for sufficiently large step sizes.

The subsystem-based decomposition trades energy consistency for direct dimensional reduction.
In particular, the loss of the port-Hamiltonian structure may impose step-size restrictions on the resulting splitting method.

\subsubsection*{\textbf{Diagonal decomposition (subsystem dynamics vs.\ coupling)}}

The diagonal decomposition follows the same subsystem-oriented perspective but separates the uncoupled subsystem dynamics from the coupling itself. It is given by, \cite{lorenz2025},
\begin{subequations}\label{eq:diag-split}
\begin{align}
\label{eq:diag-split-uncoupled} \nonumber
\splitpart{1}\colon \,\,
\bm{E}(\bm{x})\dot{\bm{x}} &=
\begin{pmatrix}
\bm{J}_1(\bm{x}_1)\!-\!\bm{R}_1(\bm{x}_1) & \bm{0}\\
\bm{0} & \bm{J}_2(\bm{x}_2)\!-\!\bm{R}_2(\bm{x}_2)
\end{pmatrix}\, \bm{z}(\bm{x})
\;+\; \begin{pmatrix}
    \bm{B}_1(\bm{x}_1) & \bm{0} \\ \bm{0} & \bm{B}_2(\bm{x}_2)
\end{pmatrix} \bm{u}(t), \\
&\hspace*{7cm}\text{(uncoupled subsystems),} 
\\
\label{eq:diag-split-coupling} 
\splitpart{2} \colon \,\,
\bm{E}(\bm{x})\dot{\bm{x}} &=
\begin{pmatrix}
\bm{0} & \bm{C}(\bm{x})\\
-\bm{C}(\bm{x})^{\top} & \bm{0}
\end{pmatrix}\, \bm{z}(\bm{x}), 
\hspace*{2.5cm}\text{(coupling).}
\end{align}
\end{subequations}
We refer to $\splitpart{1}$ as the uncoupled subproblem describing the uncoupled dynamics of the subsystems $\subsystem{1}$ and $\subsystem{2}$, and to $\splitpart{2}$ as the coupling subproblem representing the coupling effects. In contrast to the subsystem-based decomposition, both subproblems retain the port-Hamiltonian structure with the original gradient pair $(\bm E,\bm z)$. In particular, the block-diagonal matrices
\[
\bm J_{\mathrm d}(\bm x)=\operatorname{diag}\bigl(\bm J_1(\bm x_1),\bm J_2(\bm x_2)\bigr)
=-\bm J_{\mathrm d}(\bm x)^\top,
\qquad 
\bm R_{\mathrm d}(\bm x)=\operatorname{diag}\bigl(\bm R_1(\bm x_1),\bm R_2(\bm x_2)\bigr)
=\bm R_{\mathrm d}(\bm x)^\top \succeq \bm{0}
\]
are skew-symmetric and symmetric positive semidefinite.
Consequently, the uncoupled subproblem $\splitpart{1}$ is passive, whereas the coupling subproblem $\splitpart{2}$ is conservative: it transfers energy between the two subsystems without generating or dissipating energy.

The diagonal decomposition does not generally reduce the overall dimensions. However, $\splitpart{1}$ separates into two independent systems of dimensions $n_1$ and $n_2$
\begin{align*}
\bm E_i(\bm x_i)\dot{\bm x}_i = \bigl(\bm J_i(\bm x_i)-\bm R_i(\bm x_i)\bigr)\bm z_i(\bm x_i) +\bm B_i(\bm x_i)\bm u_i(t), \quad i=1,2,
\end{align*}
which can be solved independently and, in particular, in parallel. Moreover, if the coupling acts only through a low-dimensional subspace, $\splitpart{2}$ may itself admit an efficient reduced representation.
An additional advantage is that one subproblem is energy-conservative, which provides flexibility in the construction of higher-order energy-consistent splitting methods, including methods with negative sub-steps.

The diagonal decomposition can be viewed as a structure-preserving alternative to the subsystem-based decomposition: it retains the same component-level separation while isolating the coupling in a conservative subproblem. This thereby avoids the loss of the port-Hamiltonian structure associated with the subsystem-based decomposition and allows for energy-consistent splitting methods (for arbitrary step sizes $h>0$).

\begin{remark}
An important practical aspect is that both decomposition strategies, i.e., diagonal and subsystem-based decomposition, can be implemented using black-box simulators for the individual subsystems/components. This allows existing simulation codes to be incorporated without requiring access to their internal model representations. The key difference lies in how the coupling is realized.
\end{remark}

\subsubsection*{\textbf{Time-scale decomposition}}

A further source of computational savings arises when the coupled system exhibits strongly separated time scales. This situation is common in multiphysical applications; for example, electrical variables may evolve on a substantially faster time scale than thermal variables (see Section~\ref{sec:numerical_results}). Applying a single time step dictated by the fast dynamics to the entire system can then lead to unnecessary computational effort for the slow components.

Splitting methods provide a natural framework for multiple time stepping \cite{biesiadecki1993, grubmueller1991}. Consider a decomposition of the vector field
$ \bm f=\bm f^{[\mathfrak f]}+\bm f^{[\mathfrak s]}$ into fast $\bm f^{[\mathfrak f]}$ and slow dynamics $\bm f^{[\mathfrak s]}$. The impulse method, a multiple-time-stepping (multirate) variant of Strang splitting \cite{hairer2006}, applied to an autonomous system, takes the following form when the subflows are replaced by their numerical approximations $\bm \psi^{[i]}$,
\begin{equation}
\bm \Psi_h = \bm \psi^{[\mathfrak s]}_{h/2}\circ \left(\bm \psi^{[\mathfrak f]}_{h/m}\right)^m\circ\bm \psi^{[\mathfrak s]}_{h/2},
\label{eq:impulse-method}
\end{equation}
where $m\in\mathbb{N}$ is the \emph{multirate factor}. Thus, the fast subsystem is advanced with the micro-step $h/m$, whereas the slow subsystem is evaluated only on the macro time scale $h$. For non-autonomous systems, the initialization times are regarded in the subflows initializations
\begin{equation*}
\bm \Psi_{t_0+h,t_0} = \bm \psi^{[\mathfrak s]}_{t_0+h,t_0+h/2}\circ  \bm \psi^{[\mathfrak f]}_{t_0+h, t_0+(m-1)h/m} \circ \dots \circ \bm \psi^{[\mathfrak f]}_{t_0+h/m,t_0} \circ\bm \psi^{[\mathfrak s]}_{t_0+h/2,t_0}.
\end{equation*}

For the coupled pH-ODE \eqref{eq:coupledSystem}, we assume that $\bm x_1$ represents the fast variables and $\bm x_2$ the slow variables. The time-scale decomposition is then
\begin{subequations}\label{eq:mr-split}
\begin{align}
\label{eq:mr-fast-en}
\splitpart{1} \colon \,\,
\bm{E}(\bm{x})\dot{\bm{x}} &=
\underbrace{\begin{pmatrix}
        \bm{J}_1(\bm{x}_1)-\bm{R}_1(\bm{x}_1) & \bm{0}\\
        \bm{0} & \bm{0}
    \end{pmatrix}}_{=:\, \bm{J}^{[\mathfrak{f}]}(\bm{x})-\bm{R}^{[\mathfrak{f}]}(\bm{x})} \bm{z}(\bm{x})
+
\underbrace{\begin{pmatrix}
        \bm{B}_{1}(\bm{x}_1)\bm{u}_1(t)\\[0.2em]
        \bm{0}
    \end{pmatrix}}_{=:\, \bm{B}^{[\mathfrak{f}]}(t,\bm{x})\bm{u}_1(t)} && 
    \text{(fast),} \\
    \label{eq:mr-slow-en}
\splitpart{2} \colon \,\,
\bm{E}(\bm{x})\dot{\bm{x}} &=
\underbrace{\begin{pmatrix}
        \bm{0} & \bm{C}(\bm{x})\\
        -\bm{C}(\bm{x})^{\top} & \bm{J}_2(\bm{x}_2)-\bm{R}_2(\bm{x}_2)
    \end{pmatrix}}_{=:\, \bm{J}^{[\mathfrak{s}]}(\bm{x})-\bm{R}^{[\mathfrak{s}]}(\bm{x})} \bm{z}(\bm{x})
+
\underbrace{\begin{pmatrix}
        \bm{0}\\[0.2em]
        \bm{B}_{2}(\bm{x}_2)\bm{u}_2(t)
    \end{pmatrix}}_{=:\, \bm{B}^{[\mathfrak{s}]}(t,\bm{x})\bm{u}_2(t)} && 
    \text{(slow).} 
\end{align}
\end{subequations}
Both subproblems are passive pH-ODEs, in particular $\tfrac{\mathrm d}{\mathrm dt}\mathcal{H}(\bm{x}(t)) \leq \bm y_i(t)^\top \bm u_i(t)$, $i=1,2$. Consequently, the impulse method is energy-consistent.

In the fast subproblem, the slow component is frozen, $\bm E_2(\bm x_2)\dot{\bm x}_2=\bm 0$,
so that the effective dimension of $\splitpart{1}$ is $n_1$. The slow subproblem contains both the slow dynamics and the skew-symmetric coupling. Hence it generally remains of full dimension $n$, but a sparse coupling may reduce its effective size. In a best-case scenario, the coupling could be just scalar such that $\splitpart{2}$ has effectively almost the same size as without coupling, see, e.g., \cite{lorenz2025}.

The time-scale decomposition is particularly attractive when the fast subproblem is inexpensive to evaluate, e.g., because $n_1\ll n$ or because $\splitpart{1}$ is linear. In this case, the repeated evaluations within the impulse method incur only moderate additional cost while allowing the fast dynamics to be resolved more accurately. Conversely, if the fast subproblem is computationally expensive, the repeated micro-steps may offset the benefits of multiple time stepping.

The decomposition relies on the assumption that the coupling is compatible with the chosen slow--fast partition \cite{gear1984}. If the coupling itself contains fast dynamics, these contributions have to be included in the fast subproblem, which may reduce the computational advantage. A further decomposition might then be introduced to isolate the relevant coupling components and to apply an energy-consistent splitting method to the resulting three or more subproblems (hierarchical use, see Section~\ref{sec:hierarchical}).

\begin{table}[t]
\centering
\caption{Main characteristics of the decomposition approaches. The last column assumes a coupled pH-ODE with subsystems evolving on different characteristic time scales.}
\label{tab:decomp}
\begin{tabular}{lccccc}
\toprule
\textbf{Decomposition}
& \textbf{General}
& \textbf{Energy}
& \textbf{Dimensional}
& \textbf{Parallel}
& \textbf{Time-scale} \\
& \textbf{PHS}
& \textbf{consistent}
& \textbf{reduction}
& \textbf{potential}
& \textbf{separation}
\\
\midrule
Energy-associated
& Yes
& $\checkmark$
& Partial
& $\checkmark$
& No
\\
Port-based
& Yes
& $\checkmark$
& Partial
& $\checkmark$
& No
\\
Subsystem-based
& Coupled
& No
& $\checkmark$
& No
& $\checkmark$
\\
Diagonal
& Coupled
& $\checkmark$
& Partial
& $\checkmark$
& No
\\
Time-scale
& Coupled
& $\checkmark$
& Partial
& No
& $\checkmark$
\\
\bottomrule
\end{tabular}
\end{table}

\subsection{Hierarchical Use of Decomposition Approaches}\label{sec:hierarchical}

The presented decomposition appro\-aches (cf.\ Table~\ref{tab:decomp}) need not be applied exclusively or at a single level. For example, the energy- and port-based decompositions can be applied to port-Hamiltonian subproblems generated by the coupling-based decompositions (diagonal, subsystem-based, time-scale). This gives rise to a hierarchical strategy in which computational and energetic structures are exploited at different levels.

The hierarchical approach is particularly effective for the diagonal decomposition. It separates the independent subsystem/component dynamics from the conservative coupling while retaining the port-Hamiltonian structure. The resulting subsystem problems can be treated independently and in parallel, and can subsequently or recursively be decomposed according to their individual energetic structure using the energy-associated and/or port-based decompositions.
This also highlights a key distinction from the subsystem-based decomposition. The latter provides direct dimensional reduction, but its subproblems generally lose the port-Hamiltonian structure and may require step-size restrictions to maintain stability. The diagonal decomposition, in contrast, can provide comparable subsystem-level computational efficiency while ensuring energy consistency.

The hierarchical strategy can be summarized as
\begin{equation*}
\boxed{
\text{exploit coupling structure}
\quad\longrightarrow\quad
\text{preserve and exploit pH structure recursively}
}
\end{equation*}
This construction combines the computational advantages of coupling-based decompositions with the structural guarantees of energy-associated and port-based splitting, making it particularly attractive for large coupled multiphysical systems.

\section{Numerical Results}\label{sec:numerical_results}

We assess the decomposition strategies using an electro-thermal RLC network as benchmark. The numerical experiments are designed to investigate  the convergence, energy behavior, and computational cost of the single-rate splitting schemes based on the proposed decompositions, and the additional efficiency that can be gained by exploiting the inherent separation of time scales.

\subsection{Electro-Thermal Benchmark and Numerical Setup}\label{sec:benchmark}

We consider an electro-thermal RLC network with $N$ building blocks, as described in Appendix~\ref{app:thermal-electric-modeling}. The coupled system $\system$ consists of an electrical subsystem $\subsystem{1}$ for node potentials and inductor currents with $n_1=2N+1$ and a thermal subsystem $\subsystem{2}$ for the lumped entropies (temperatures) with $n_2=N$. The coupled pH-ODE \eqref{eq:benchmark} has hence dimension $n=3N+1$.
For the numerical simulations, we consider identical blocks using the following parameter values in SI units, 
\begin{equation*}\label{eq:parameters}
\begin{aligned}
    &C_0 = 10^{-3},\; \; C = 10^{-4},\; \; \; L = 10^{-2},\; \; \; R = 1500,\; \; \; R_0 = 2\cdot 10^{-1},\; \; \; \alpha_1 = 5\cdot 10^{-1},\; \; \; \alpha_2 = 10^{-3},\\
    &T_{\mathrm{env}} = T_{\mathrm{ref}} = 300,\; \; \;  M_i = 10^{-2}, \; \; \; \Gamma_i = 2\cdot 10^{-3},\; \; \; \Lambda_{i,i+1} = 2\cdot 10^{-3} \,\, \forall i,
\end{aligned}
\end{equation*}
as well as $\imath(t) = 3\sin(2\pi\cdot 10^3t)$. The initial value $\bm{x}_0 \in \R^{3\nbuild +1}$ is set to
\begin{align*}
    x_{0,i} =
\begin{cases}
1, & i = 1,\\
0.1, & i = 2k,\; k=1,\dots,\nbuild,\\
0, & \text{otherwise}.
\end{cases}
\end{align*}
Unless otherwise stated, we use $N=100$.

\begin{figure}[t]
    \begin{subfigure}{0.496\textwidth}
  \centering
    \includegraphics{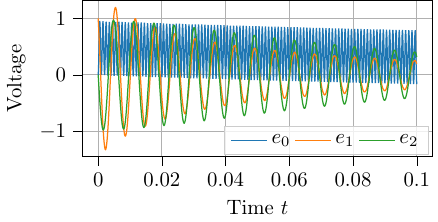}
  \caption{Voltage components.}
  \label{fig:voltage}
\end{subfigure}
\begin{subfigure}{0.496\textwidth}
  \centering
  \includegraphics{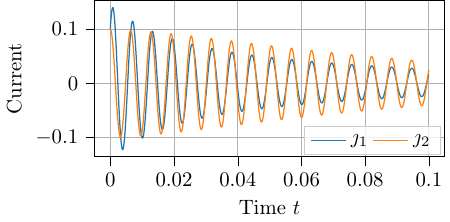}
  \caption{Current components.}
  \label{fig:current}
\end{subfigure}
\begin{subfigure}{0.496\textwidth}
  \centering
  \includegraphics[width=\textwidth]{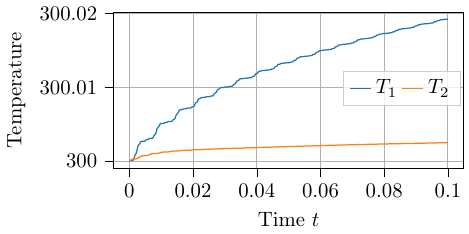}
  \caption{Temperature components.}
  \label{fig:temperature}
\end{subfigure}
    \caption{Electro-thermal network with $N=2$: voltages, currents, and temperatures for time $t\in [0,0.1]$ (reference computed with \texttt{solve\_ivp}).}
    \label{fig:Dynamics}
 \end{figure}

Figure~\ref{fig:Dynamics} illustrates the temporal evolution of the voltages (node potentials), currents, and temperatures for $N=2$. The voltage and current variables exhibit rapid variations and pronounced oscillations, whereas the temperature evolves on a substantially slower time scale. In particular, the node potential $e_0$ shows highly oscillatory behavior induced by the external input. The pronounced separation between the fast electrical and slow thermal dynamics makes the benchmark particularly suitable for the proposed decomposition strategies. In particular, it provides a natural setting in which the multirate potential of the coupled (multiphysical) system can be exploited.

For the numerical simulations presented in this section, the proposed decomposition strategies are embedded into the symmetric second-order Strang splitting \eqref{eq:Strang} and its multiple-time-stepping extension (impulse method) \eqref{eq:impulse-method}. The exact subflows are approximated by the second-order discrete gradient method \eqref{eq:DGM} with the Gonzalez discrete gradient \eqref{eq:Gonzalez_DG} (cf.\ Appendix~\ref{app:DGM}). Thus, the overall numerical integration schemes are of order $p=2$. The implementation is done in Python~3.14.0. The nonlinear systems arising from the implicit discrete gradient method are solved by Newton's method with an absolute tolerance of $10^{-8}$ and a maximum of 20 iterations. The respective Jacobians are approximated by two-point finite differences using \texttt{scipy.optimize.\_numdiff.approx\_derivative(\(\cdot\), method="2-point")}. Linear systems are solved by a direct solver.
The reference solutions are computed via \texttt{solve\_ivp} using a BDF method with absolute tolerance of $10^{-12}$ and relative tolerance of $10^{-12}$.

\subsection{Accuracy, Energy Behavior, and Computational Efficiency} \label{sec:single_rate}

We investigate the performance of the decomposition strategies, regarding accuracy, computational efficiency and energy behavior of the associated Strang splitting schemes.

For the numerical comparisons, we consider the integration schemes listed in Table~\ref{tab:integrator-strategies}. The labels indicate the decomposition and the ordering of the subproblems within the splitting. In particular, \texttt{JR} and \texttt{RJ} refer to the energy-associated decomposition, \texttt{PB1} and \texttt{PB2} to the port-based decomposition, \texttt{Dim1} and \texttt{Dim2} to the subsystem-based decomposition, \texttt{DO} and \texttt{OD} to the diagonal decomposition, and \texttt{TS} to the time-scale decomposition. The method \texttt{DG} is the Gonzalez discrete gradient method directly applied to the original pH-ODE~\eqref{eq:benchmark}. The main structural and computational characteristics of the decompositions are summarized in Table~\ref{tab:decomposition-comparison}.  Apart from the energy consistency, it states the linearity and effective dimension of the resulting subproblems. The latter refers to the number of active state variables and hence provides an indication of the computational complexity. The simulation results shown for $N=100$ are representative of the overall behavior. For $N=10$ and $N=1000$, the work-precision diagrams exhibit the same relative performance of the methods, while the absolute computational times increase with $N$.

\begin{table}[tb!]
    \centering
    \caption{Integration schemes: $\bm \Psi_h=\bm \psi_{h/2}^{[\mathfrak a]}\circ (\bm \psi_{h/m}^{[\mathfrak b]})^m \circ \bm \psi_{h/2}^{[\mathfrak a]}$ with $m=1$ in Strang-splitting (single-rate). The listed $m$ gives the multirate factor used in the multiple-time-stepping experiments (cf.\ Section~\ref{sec:multirate}).}
    \label{tab:integrator-strategies}
    \begin{tabular}{llllr}
        \toprule 
        \textbf{ID} & \textbf{Decomposition} & $\bm{\splitpart{\mathfrak{a}}}$ & $\bm{\splitpart{\mathfrak{b}}}$  & $\bm{m}$ \\
        \midrule 
         \texttt{JR} & Energy-associated & \phantom{1}\eqref{eq:JR-split-conservative} conservative & \phantom{1}\eqref{eq:JR-split-dissipative} passive & $5$ \\
         \texttt{RJ} & Energy-associated & \phantom{1}\eqref{eq:JR-split-dissipative} passive & \phantom{1}\eqref{eq:JR-split-conservative}  conservative & $1$ \\
        \texttt{PB1} & Port-based & \eqref{eq:port-split-int}  internal & \eqref{eq:port-split-ext} external & $1$\\
        \texttt{PB2} & Port-based & \eqref{eq:port-split-ext} external & \eqref{eq:port-split-int} internal & $1$\\
        \texttt{Dim1} & Subsystem-based & \eqref{eq:dim-split-1} $\subsystem{1}$ & \eqref{eq:dim-split-2} $\subsystem{2}$ & $1$ \\
         \texttt{Dim2} & Subsystem-based & \eqref{eq:dim-split-2} $\subsystem{2}$ & \eqref{eq:dim-split-1} $\subsystem{1}$ & $30$\\
        \texttt{DO} & Diagonal & \eqref{eq:diag-split-uncoupled} uncoupled & \eqref{eq:diag-split-coupling} coupling & $1$\\
        \texttt{OD} & Diagonal & \eqref{eq:diag-split-coupling} coupling & \eqref{eq:diag-split-uncoupled} uncoupled & $15$\\
         \texttt{TS} & Time-scale & \eqref{eq:mr-slow-en} slow & \eqref{eq:mr-fast-en} fast & $150$\\
          \texttt{DG} & \multicolumn{4}{l}{\emph{Gonzalez discrete gradient method without any decomposition}}   \\
         \bottomrule 
    \end{tabular}
\end{table}
\begin{table}[tb!]
\centering
\caption{Decomposition characteristics for electro-thermal network: information about subproblems regarding linearity ($\ell$ linear, n$\ell$ nonlinear), effective dimension, and energy consistency of associated splitting. An entry $\dots \& \dots$ indicates that $\splitpart{i}$ separates into two independent systems.}
\label{tab:decomposition-comparison}
\begin{tabular}{llllc}
\toprule
\makecell[c]{\textbf{Decomposition}\\\textbf{}}
&
\makecell[l]{\textbf{Subproblem}\\ $\splitpart{1}$, $\splitpart{2}$ }
&
\makecell[l]{\textbf{Linearity}\\\textbf{}}
&
\makecell[l]{\textbf{\,Effective}\\\textbf{dimension}}
&
\makecell[c]{\textbf{Energy}\\\textbf{consistent}}
\\
\midrule
\multirow{2}{*}{Energy-associated}
&conservative&n$\ell$&$3\nbuild+1$&\multirow{2}{*}{\makecell[c]{\checkmark}}\\
&passive&$\ell $ \&  n$\ell$&$(\nbuild+1)$ \&  $\nbuild$&\\
\midrule
\multirow{2}{*}{Port-based}
&internal&n$\ell$&$3\nbuild+1$&\multirow{2}{*}{\makecell[c]{\checkmark}}\\
&external&$\ell $ \& n$\ell$&1 \& $\nbuild$&\\
\midrule
\multirow{2}{*}{Subsystem-based}
&$\subsystem{1}$&n$\ell$&$2\nbuild+1$&\multirow{2}{*}{\makecell[c]{No}}\\
&$\subsystem{2}$&n$\ell$&$\nbuild$&\\
\midrule
\multirow{2}{*}{Diagonal}
&uncoupled&$\ell$ \&  n$\ell$&$(2\nbuild+1)$ \& $\nbuild$&\multirow{2}{*}{\makecell[c]{\checkmark}}\\
&coupling&n$\ell$&$2\nbuild+1$&\\
\midrule
\multirow{2}{*}{Time-scale}
&fast&$\ell$&$2\nbuild+1$&\multirow{2}{*}{\makecell[c]{\checkmark}}\\
&slow&n$\ell$&$2\nbuild+1$&
\\
\bottomrule
\end{tabular}
\end{table}

\begin{figure}[t]
    \centering
    \includegraphics[]{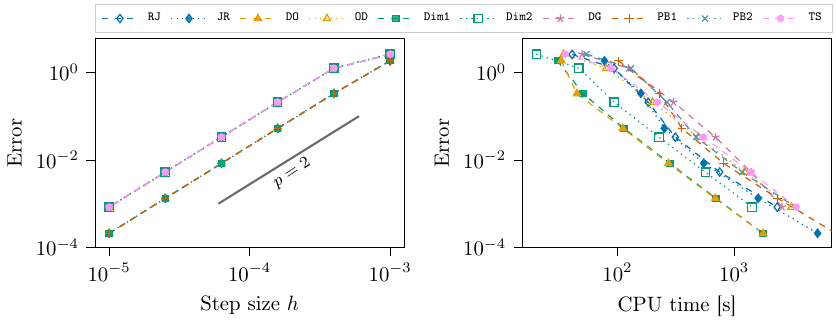}
    \caption{Decomposition strategies in Strang splitting applied to electro-thermal network, $N=100$. Left: discrete $L^2([0,0.1])$-error in time versus step size $h$. Right: error versus CPU time (in seconds).}
    \label{fig:decomposition_comparison}
\end{figure}

Figure~\ref{fig:decomposition_comparison} shows the discrete $L^2([0,0.1])$-error in time as a function of the step size $h$ and the corresponding work-precision diagram. All decomposition strategies exhibit the expected second-order convergence. The methods mainly differ in their error constants and computational costs. The convergence results reveal two groups of methods with different error levels (Figure~\ref{fig:decomposition_comparison}(left)); in particular \texttt{Dim1}, \texttt{DO}, \texttt{JR}, and \texttt{PB1} are of higher accuracy. In terms of computational efficiency, the subsystem-based variant \texttt{Dim1} and the diagonal variant \texttt{DO} perform best, they are particularly competitive. This behavior is consistent with the reduced effective dimensions and the possibility of solving independent subsystem problems separately or in parallel (cf.\ Table~\ref{tab:decomposition-comparison}). For comparison, the unsplit Gonzalez discrete gradient method \texttt{DG} is also included. Most of the splitting approaches provide a more favorable work-precision trade-off than the direct application of the discrete gradient method, demonstrating the computational benefit of exploiting the decomposition structure for the complex multiphysical benchmark. 

\begin{figure}[b]
    \centering
    \includegraphics[]{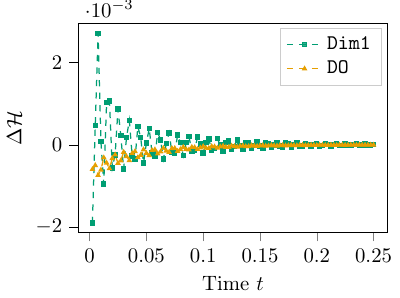}\hfill\includegraphics[]{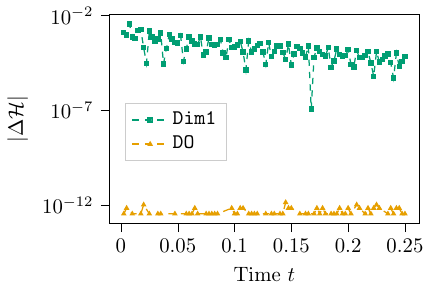}
    \caption{Energy behavior $\Delta \mathcal{H}_{k+1} \coloneqq (\mathcal{H}(\bm{x}_{k+1}) - \mathcal{H}(\bm{x}_k))/h$, $h=2.5 \cdot 10^{-3}$, of \texttt{DIM1} and \texttt{DO} for the network, $N=2$, without external input. The tolerance of the Newton iteration is $\mathtt{TOL} = 10^{-12}$.
  Left: energy-dissipative system, $R=1500$ (default). Right: numerically energy-conserving system, $R = 10^{14}$.}
    \label{fig:energy_behavior}
\end{figure}

The comparison between the coupling-based decompositions is of particular interest. 
As for the energy consistency, Figure~\ref{fig:energy_behavior} compares the discrete energy behaviors of the subsystem-based decomposition \texttt{Dim1} and the diagonal decomposition \texttt{DO}. To study the energy behavior, all external inputs are set to zero ($\imath \equiv 0$, $\Gamma_i = 0$ for all $i$). The resulting system $\system$ is dissipative and satisfies $\tfrac{\mathrm d}{\mathrm dt} \mathcal{H}(\bm x(t))\leq 0$, so that the exact Hamiltonian is non-increasing, $\mathcal{H}(\bm{\varphi}_{t,0}(\bm{x}_0)) \leq \mathcal{H}(\bm{x}_0)$ for all $t>0$.
The diagonal decomposition \texttt{DO} preserves the qualitative dissipative behavior: the discrete energy difference remains non-positive for the iterates $\bm x_k$, $\mathcal{H}(\bm x_{k+1})-\mathcal{H}(\bm x_k)\leq0$. In contrast, the subsystem-based decomposition \texttt{Dim1} does not satisfy the dissipation inequality at the discrete level, with the corresponding values oscillating around zero, see Figure~\ref{fig:energy_behavior}(left).
To further assess energy conservation, we additionally reduce the dissipation to a negligible level by choosing the resistance parameter sufficiently large, $R=10^{14}$. The resulting system is numerically energy-conservative. In this setting, \texttt{DO} yields an energy-conserving numerical method, with energy errors on the order of the tolerance used in the Newton iteration (Figure~\ref{fig:energy_behavior}(right)). This follows from the fact that both subflows are energy-conservative and approximated by the energy-conserving discrete gradient method. In contrast, \texttt{Dim1} produces subproblems that are not themselves port-Hamiltonian and hence do not preserve the energy at the subproblem level. The resulting splitting method consequently fails to conserve the Hamiltonian, as illustrated in Figure~\ref{fig:energy_behavior}(right).

These results demonstrate the practical relevance of the structural distinction between the subsystem-based and diagonal decompositions. 
The subsystem-based decomposition provides a direct dimensional reduction, but the resulting subproblems generally lose the port-Hamiltonian structure. Consequently, the corresponding splitting schemes do not possess the general energy-consistency guarantee and step-size restrictions are necessary. In contrast, the diagonal decomposition separates the uncoupled subsystem dynamics from the conservative coupling while preserving the port-Hamiltonian structure of both subproblems. Thus, the diagonal decomposition combines the computational advantages of subsystem-level separation with the structural properties required for energy-consistent integration (Lemma~\ref{lemma: energy-consistent splitting}). The work-precision results indicate that this can be achieved without a significant loss in computational efficiency compared with the subsystem-based decomposition.

\begin{figure}[t]
    \centering
     \includegraphics[]{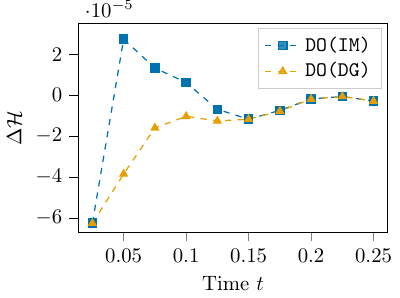}
       \caption{Energy behavior $\Delta \mathcal{H}_{k+1} \coloneqq (\mathcal{H}(\bm{x}_{k+1}) - \mathcal{H}(\bm{x}_k))/h$, $h= 2.5 \cdot 10^{-2}$, of \texttt{DO} with subflow approximation via Gonzalez discrete gradient method \texttt{DO(DG)} and implicit midpoint rule  \texttt{DO(IM)} for the network, $N=2$, without external input, $R = 10^{4}$, $M_i = 3\cdot 10^{-6}$, and initial value $\bm{x}_0 = (0.1,-0.5,0.1,0.5,-0.5)^\top$. The tolerance of the Newton iteration is $\mathtt{TOL} = 10^{-12}$.}
    \label{fig:energy_discrete_gradient}
\end{figure}

The role of energy-consistent subflow approximations in preserving energy consistency of the overall approximation has been analyzed qualitatively in Proposition~\ref{prop: Energyconsistentsplit}. A quantitative confirmation is provided in Figure~\ref{fig:energy_discrete_gradient}, where the subflows of the diagonal decomposition $\texttt{DO}$ in a dissipative setting are approximated using the Gonzalez discrete gradient method \texttt{DO(DG)} and the implicit midpoint rule \texttt{DO(IM)}. Although the implicit midpoint rule is a second-order A-stable scheme, it causes positive discrete energy differences at several time points, as it is not energy-consistent for the nonlinear Hamiltonian considered in the benchmark.

\subsection{Exploiting Time-Scale Separation}\label{sec:multirate}

The electro-thermal benchmark exhibits a pronounced separation of time scales, as demonstrated in Figure~\ref{fig:Dynamics}. We therefore investigate whether this structure can be exploited computationally by multiple time stepping.

We consider the impulse method based on the time-scale decomposition. The fast subproblem $\splitpart{1}$ (electric dynamics) is integrated with the micro-step $h/m$, whereas the slow subproblem $\splitpart{2}$ (thermal dynamics and coupling) is advanced with the macro-step $h$. Figure~\ref{fig:multirate_factor} shows the resulting convergence behavior and work-precision diagrams for the multirate factors $m\in\{1,10,50,150,200,300\}$. For all values of $m$, the expected second-order convergence is retained. Increasing $m$ improves the computational efficiency because the fast dynamics can be resolved on a finer time scale without requiring additional evaluations of the slow subproblem. For the present benchmark, a multirate factor of approximately $m=150$ provides the most favorable work-precision performance. For larger values of $m$, the additional cost associated with the repeated fast subproblem evaluations increasingly offsets the savings obtained from reducing the number of slow subproblem evaluations. The optimal multirate factor is problem-dependent and reflects the relative time scales and computational costs of the individual subproblems.

\begin{figure}[t]
    \centering
    \includegraphics{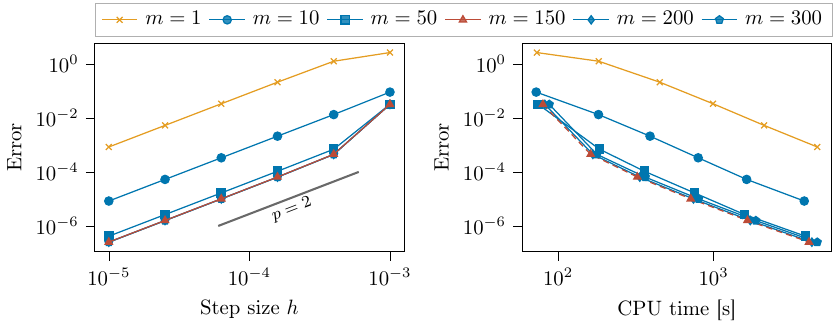}
    \caption{Impulse method based on the time-scale decomposition for different multirate factors $m\in\{1,10,50,150,200,300\}$. Left: discrete $L^2([0,0.1])$-error in time versus step size $h$. Right: error versus CPU time.}
    \label{fig:multirate_factor}
\end{figure}

We apply the same multiple-time-stepping principle to the remaining decomposition strategies. For each integrator, the multirate factor is selected according to its best work-precision performance. The resulting values are reported in Table~\ref{tab:integrator-strategies}.
Figure~\ref{fig:multirate_comparison} compares the resulting multirate variants. All methods retain their second-order convergence. In particular, the variants \texttt{Dim2}, \texttt{OD} and \texttt{TS} benefit substantially from multiple time stepping compared with their corresponding single-rate versions in Figure~\ref{fig:decomposition_comparison}(right). The multirate variant \texttt{TS} provides the most favorable computational performance for the present benchmark, which is not surprising, as the time-scale decomposition is specifically designed to exploit differences in the characteristic time scales of the coupled subsystems.
It requires only 22.1\% and 9.8\% of the computational cost of \texttt{OD} and \texttt{Dim2}, respectively, to achieve an error of $\mathcal{O}(10^{-4})$ in the discrete $L^2$-norm.
Importantly, the use of multiple time stepping does not compromise the structural properties of the underlying decomposition. 
Energy consistency is inherited by the corresponding multiple time stepping (multirate splitting) scheme under the assumptions of Lemma~\ref{lemma: energy-consistent splitting}.

\begin{figure}[t]
    \centering
    \includegraphics[]{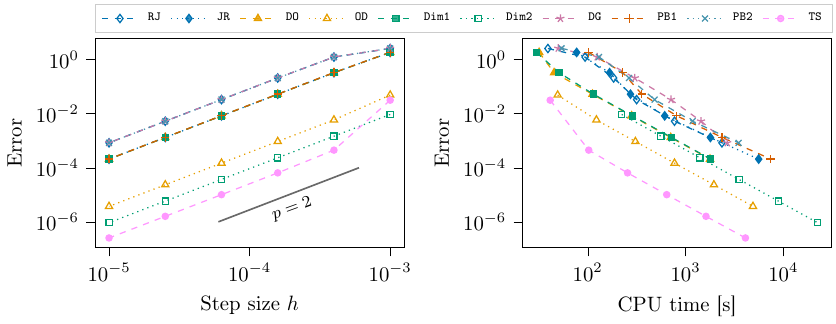}
    \caption{Decomposition strategies in impulse method with the multirate factors listed in Table~\ref{tab:integrator-strategies}. Left: discrete $L^2([0,0.1])$-error in time versus step size $h$. Right: error versus CPU time.}
    \label{fig:multirate_comparison}
\end{figure}

The numerical results demonstrate that in the present benchmark the separation of electrical and thermal time scales can be effectively exploited to reduce the computational effort. The time-scale multirate approach provides a substantial efficiency gain without sacrificing convergence order or energy consistency.

\section{Conclusion}\label{sec:conclusion}

In this work, we studied splitting methods for (coupled) port-Hamiltonian ODEs and analyzed how the choice of decomposition affects both structural properties and computational efficiency. 

Three main conclusions emerge. First, the choice of decomposition is crucial for preserving the energetic properties of the system. While the well-established subsystem-based decomposition (component-wise partitioning) can provide computationally efficient lower-dimensional subproblems, it generally destroys the port-Hamiltonian structure at the subflow level and therefore does not guarantee energy consistency. Controlling the resulting numerical energy behavior may consequently require step-size restrictions.
Second, structural preservation and computational efficiency need not constitute competing objectives. For coupled systems, the proposed diagonal decomposition exploits the coupling structure, enables dimension reduction and parallelization, while retaining the port-Hamiltonian structure of the subproblems and thereby ensuring energy-consistent splitting schemes. It thus combines competitive computational performance with the desired energetic properties and is particularly attractive for large-scale coupled systems.
Third, well-separated time scales can be effectively exploited by combining the time-scale decomposition with multiple time stepping. For the electro-thermal benchmark considered here, this yields a substantial additional efficiency gain while retaining the convergence and energetic properties of the underlying splitting method. The achievable gain is problem-dependent and reflects the separation of time scales and the relative computational costs of the subproblems.

Overall, an appropriate decomposition can exploit both, computational structure and energetic structure, potentially in a hierarchical manner. This suggests a natural hierarchical strategy in which coupling and time-scale structures are exploited first and the resulting subproblems are further decomposed using energy-associated and/or port-based approaches. Investigating such hierarchical splitting strategies for coupled pH-ODEs, together with their numerical stability properties, constitutes a natural direction for future work.

\subsection*{Acknowledgements} This work was partially funded by the Deutsche Forschungsgemeinschaft (DFG, German Research Foundation), Project-ID 531152215, CRC 1701 Port-Hamiltonian Systems.

\appendix
\numberwithin{equation}{section}
\numberwithin{figure}{section}
\numberwithin{table}{section}
\numberwithin{definition}{section}
\numberwithin{example}{section}
\numberwithin{remark}{section}
\numberwithin{lemma}{section}

\section{Discrete Gradient Methods} \label{app:DGM}

Discrete gradient methods constitute a class of energy-consistent numerical integration schemes for port-Hamiltonian systems \cite{kinon2026}.

\begin{definition}
\label{def:general-dg}
Let $\mathcal{H}:\R^n\rightarrow \R$ be sufficiently smooth. The function $\bar\nabla\mathcal H \colon \R^n\times\R^n\to\R^n$ is said to be a \emph{discrete gradient} if it satisfies
\begin{enumerate}
    \item[i)] $\bar\nabla\mathcal H(\bm{x}^\prime,\bm{x})^\top (\bm{x}^\prime - \bm{x}) = \mathcal H(\bm{x}^\prime)-\mathcal H(\bm{x})$ for all $\bm{x},\bm{x}^\prime \in\R^n$,
    \item[ii)] $\bar\nabla\mathcal H(\bm{x},\bm{x})=\nabla\mathcal H(\bm{x})$.
\end{enumerate} 
We call it a \emph{second-order discrete gradient} if
$$\bar\nabla\mathcal H(\bm{x}^\prime,\bm{x})=\nabla\mathcal{H}(\tfrac{\bm{x}^\prime+\bm{x}}{2})+\mathcal{O}(\|\bm{x}^\prime-\bm{x}\|^2).$$
\end{definition}

\begin{example}
    A prominent second-order discrete gradient is the Gonzalez discrete gradient \cite{gonzalez1996}
        \begin{align}\label{eq:Gonzalez_DG}
        \bar \nabla \mathcal{H}(\bm{x}^{\prime},\bm{x}) = \begin{cases} 
            \nabla \mathcal{H}(\bar{\bm{x}}) + \frac{\mathcal{H}(\bm{x}^{\prime}) - \mathcal{H}(\bm{x}) - \nabla \mathcal{H}(\bar{\bm{x}})^{\top} (\bm{x}^\prime - \bm{x})}{\lVert  (\bm{x}^\prime - \bm{x})\rVert_2^2} (\bm{x}^\prime - \bm{x}), & \bm{x}^{\prime} \neq \bm{x}, \\
            \nabla \mathcal{H}(\bm{x}), & \bm{x}^{\prime} = \bm{x},
    \end{cases}
    \end{align}
    where $\bar{\bm x} = (\bm x^{\prime}+\bm x)/2$ denotes the midpoint and $\lVert \cdot \rVert_2$ the Euclidean norm in $\R^n$. 
Further examples include the mean-value discrete gradient \cite{harten1983} and the symmetrized Itoh--Abe discrete gradient \cite{itoh1988,eidnes2022}.
\end{example}

For pH-ODEs~\eqref{eq:pH-ODE}, we focus on second-order discrete gradient methods of the form
\begin{equation}\label{eq:DGM}
\begin{aligned}
    \bm E(\bar{\bm x})(\bm{x}_{1} - \bm{x}_0) &= h\big[(\bm J(\bar{\bm x})-\bm R(\bar{\bm x}))\bar{\bm z}(\bm{x}_{1},\bm{x}_{0})+\bm B(\bar{\bm x})\bm u(\bar t)\big], \\ \bar{\bm z}(\bm{x}_{1},\bm{x}_{0}) &= \bm{E}(\bar{\bm x})^{-\top}\bar\nabla\mathcal H(\bm{x}_{1},\bm{x}_{0}),
\end{aligned}
\end{equation}
where $\bar{\nabla}\mathcal{H}(\bm{x}_{1},\bm{x}_0)$ is a second-order discrete gradient, $\bar{\bm{x}} = (\bm{x}_{1} + \bm{x}_0)/2$, and $\bar{t}=t_0+h/2$.

\begin{lemma}\label{lemma:DGM}
 Let $\system$ be a pH-ODE \eqref{eq:pH-ODE} with Hamiltonian $\mathcal{H}\in \mathcal{C}^3(\mathbb{R}^n,\mathbb{R})$, flow matrix function $\bm E \in \mathcal{C}^2(\mathbb{R}^n,\mathbb{R}^{n\times n})$, port function $\bm B \in \mathcal{C}^2(\mathbb{R}^n,\mathbb{R}^{n\times m})$, input $\bm u \in \mathcal{C}^2([t_0,T],\mathbb{R}^m)$ and $\bm{f}\in\mathcal{C}^1([t_0,T]\times \mathbb{R}^n,\mathbb{R}^n)$.
 Let $\bm x_1=\bm \Psi_{t_1,t_0}(\bm x_0)$  be the numerical approximation of a second-order discrete gradient method \eqref{eq:DGM} at $t_1=t_0+h$, $h>0$. Then,
\begin{align*}
\mathcal H(\bm x_{1})-\mathcal H(\bm x_0) \leq h\,\bar{\bm z}(\bm{x}_{1},\bm{x}_0)^\top \bm{B}(\bar{\bm x})\, \bm{u}(\bar{t}), 
\end{align*}
with $\bar{\bm z}(\bm{x}_{1},\bm{x}_0) = \bm{E}(\bar{\bm x})^{-\top} \bar\nabla \mathcal H(\bm x_{1},\bm x_0)$, $\bar{\bm x} = (\bm x_{1}+\bm x_0)/2$, and $\bar{t} = t_0 + h/2$. The scheme $\bm\Psi$ is energy-consistent of order $q=2$.
\end{lemma}

\begin{proof}
The definition of the discrete gradient and \eqref{eq:DGM} directly yield the discrete power balance
\begin{align*}
\mathcal{H}(\bm{x}_{1})-\mathcal{H}(\bm{x}_0) = -h\,\bar{\bm{z}}(\bm{x}_{1},\bm{x}_0)^\top \bm{R}(\bar{\bm x})\bar{\bm z}(\bm{x}_{1},\bm{x}_0) + h\,\bar{\bm{z}}(\bm{x}_{1},\bm{x}_0)^\top \bm{B}(\bar{\bm x})\bm{u}(\bar{t}),
\end{align*}
where we identify the dissipated and supplied energy distributions as
\begin{equation*}
    \mathcal{D}_h= -h\,\bar{\bm z}(\bm{x}_{1},\bm{x}_0)^\top \bm{R}(\bar{\bm x})\bar{\bm z}(\bm{x}_{1},\bm{x}_0) \leq 0,\qquad 
    \mathcal{S}_h = h\,\bar{\bm{z}}(\bm{x}_{1},\bm{x}_0)^\top \bm{B}(\bar{\bm x})\bm{u}(\bar{t}).
\end{equation*}
Consistency of the method gives
$\bar{\bm x} \to\bm x_0$ and
$\bar{\bm z}(\bm{x}_{1},\bm{x}_0)\to\bm z(\bm x_0)$ as $h\to0$. Consequently,
\begin{equation*}
    \lim_{h\to0}\frac{\mathcal{D}_h}{h} = -\bm z(\bm x_0)^\top \bm R(\bm x_0)\bm z(\bm x_0), \qquad \lim_{h\to0}\frac{\mathcal{S}_h}{h} = \bm y(t_0)^\top\bm u(t_0).
\end{equation*}
Let $\bm x(t)=\bm \varphi_{t,t_0}(\bm x_0)$ denote the exact solution of \eqref{eq:pH-ODE}. 
The order of consistency carries over to the order of energy consistency under the stated regularity assumptions. Since the discrete gradient is second-order, i.e., $\bar\nabla \mathcal H(\bm x_{1},\bm x_0) = \nabla\mathcal H(\bar{\bm x})+\mathcal O(h^2)$, it holds
\begin{align*}
\bar{\bm z}(\bm{x}_{1},\bm{x}_0) = \bm z(\bar{\bm x})+\mathcal O(h^2), \quad \text{ where } \bar{\bm x} = \bm x(\bar t)+\mathcal O(h^2).
\end{align*}
With the regularity of the system functions we find
\begin{equation*}
   \bar{\bm{z}}(\bm{x}_1,\bm x_0)^\top \bm R(\bar{\bm x})\bar{\bm z}(\bm x_1,\bm x_0) =\bm z(\bm x(\bar t))^\top \bm R(\bm x(\bar t))\bm z(\bm x(\bar t)) + \mathcal{O}(h^2), \quad  \bm{B}(\bar{\bm x})^\top \bar{\bm{z}}(\bm{x}_{1},\bm{x}_0)= \bm y(\bar t)+\mathcal O(h^2),
\end{equation*}
therefore,
\begin{equation*}
    \mathcal D_h = -h\,\bm z(\bm x(\bar t))^\top \bm R(\bm x(\bar t))\bm z(\bm x(\bar t))+ \mathcal{O}(h^3), \quad \mathcal{S}_h = h\,\bm y(\bar t)^\top\bm u(\bar t) +\mathcal O(h^3).
\end{equation*}
On the other hand, the midpoint quadrature rule yields
\begin{align*}
    \int_{t_0}^{t_0+h} - \bm z(\bm x(\tau))^\top \bm R(\bm x(\tau))\bm z(\bm x(\tau))\,\diff \tau \,&{=}\, -h\,\bm z(\bm x(\bar t))^\top \bm R(\bm x(\bar t))\bm z(\bm x(\bar t))+ \mathcal{O}(h^3),\\ 
     \int_{t_0}^{t_{0}+h} \bm y(\tau)^\top\bm u(\tau)\,\diff \tau \,&{=}\, h\,\bm y(\bar t)^\top\bm u(\bar t) +\mathcal O(h^3).
\end{align*}
\end{proof}
In the numerical simulations, we embed the discrete gradient method \eqref{eq:DGM} equipped with the Gonzalez discrete gradient \eqref{eq:Gonzalez_DG} into the Strang splitting scheme \eqref{eq:Strang} and its multiple-time-stepping extension \eqref{eq:impulse-method}, yielding energy-consistent second-order splitting approaches for pH-ODEs (cf.\ Proposition~\ref{prop: Energyconsistentsplit}). Note that for quadratic Hamiltonians, $\mathcal{H}(\bm{x}) = \tfrac{1}{2} \bm{x}^\top \bm{Q} \bm{x}$, $\bm Q=\bm Q^\top \succ 0$, the Gonzalez discrete gradient reduces to $\bar{\nabla} \mathcal{H}(\bm{x}^\prime,\bm{x}) = \bm{Q}\frac{\bm{x}^\prime + \bm{x}}{2}$ and the symmetric second-order discrete gradient method \eqref{eq:DGM} becomes the implicit midpoint rule (one-stage Gauss collocation scheme). The resulting linear systems can be solved efficiently and in a structure-preserving manner using iterative Krylov subspace methods, such as Q-Arnoldi-type approaches, \cite{maier2025}.

\section{Electro-thermal Modeling} \label{app:thermal-electric-modeling}

We set up a coupled electro-thermal system within the port-Hamiltonian framework. 
Electrical and thermal subsystems are formulated as individual PHS and then interconnected through energy-conserving internal ports. This yields a coupled port-Hamiltonian model that explicitly accounts for temperature-dependent electrical parameters, Joule heating, heat conduction, and heat exchange with the environment. Naturally, the model provides a power balance as well as preserves passivity, and serves with its rich structure as the test basis for the decomposition strategies and splitting approaches.

\subsection{Coupled Electro-Thermal Port-Hamiltonian Model}
\subsubsection*{\textbf{Electric network model}}

We consider an electrical RLC network consisting of resistors $G$, capacitors $C$, inductors $L$, and independent voltage and current sources $V$ and $I$. 
Their interconnections are described by element-specific incidence matrices $\bm{A}_X \in \{-1,0,1\}^{\nnodes\times b_X}$, $X\in \{\resist , L, C, V, I\}$, where $\nnodes\in\mathbb{N}$ denotes the number of circuit nodes (including ground) and $b_X\in\mathbb{N}$ the number of branches of type $X$.
Using modified nodal analysis (MNA), the network is represented by the pH-DAE, see, e.g., \cite{bartel2024,bartel2018}, 
\begin{align}\label{eq:MNA-model}
 \begin{pmatrix}
  \bm{A}_C \capM \bm{A}_C^{\top} &  \\
                 & \bm{L}\\
                 & & \bm{0}
  \end{pmatrix} \dot{\bm{x}}_C
  = 
  &
   	\begin{pmatrix}
   		\bm{A}_{\resist} \resistM \bm{A}_{\resist}^{\top}    & -\bm{A}_L & -\bm{A}_V \\ 
		\bm{A}_L ^\top&   \bm{0} \\
		\bm{A}_V^\top &      & \bm{0}
	\end{pmatrix} 
    \bm{x}_C 
    + \begin{pmatrix} -\bm{A}_I & \bm{0} \\
   	                              \bm{0}    & \bm{0} \\
   	                              \bm{0}    & -\bm{I}  \end{pmatrix} 
                                  \underbrace{
                                  \begin{pmatrix} \bm{\imath}(t) \\
                                    \bm{v}(t) 
                                  \end{pmatrix}}_{={\bm{u}}_C^{\ex}(t)},                              
\end{align}
where $\capM,\bm{L},\resistM$ are the positive definite parameter matrices of the capacitors, inductors and resistors, respectively.
The state is
\[
 \bm{x}_C = (\bm{e}^{\top}\!\!,\; \bm{\jmath}_L^{\top}\!,\; \bm{\jmath}_V^{\top} )^{\top} 
\]
 with node potentials $\bm{e}(t)\in\mathbbm{R}^{\nnodes}$, inductor currents $\bm{\jmath}_L(t)\in \mathbbm{R}^{b_L}$, and voltage-source currents  $\bm{\jmath}_V(t)\in \mathbbm{R}^{b_V}$. The external input consists of the current and voltage sources $\bm u_C^{\ex}(t)=(\bm{\imath}^\top, \bm{v}^\top)^\top(t) \in \mathbb{R}^{b_I+b_V}$.
 
To account for electro-thermal effects, we distinguish between thermally relevant $\resistT$ and thermally irrelevant resistors $\resistN$,
\begin{equation*}
  \bm{A}_{\resist} = (\bm{A}_{\resistT},\, \bm{A}_{\resistN} ), \qquad \resistM = \operatorname{blkdiag}(\resistTM,\resistNM).
\end{equation*}
For the thermally relevant resistors we assume two-terminal elements and associate a single lumped temperature $T_i$ with each resistor. Their conductance matrix is therefore diagonal,
\[
\resistTM=\resistTM(T) = \text{diag}\Bigl(\tfrac{1}{R_1 (T_1)},\dotsc, \tfrac{1}{R_{\nT} (T_{\nT})}\Bigr),
\]
where $R_i(T_i)$ denotes the temperature-dependent electric resistance. The temperature dependence can be modeled, e.g., quadratically,  $R (T)= R_0 + \alpha_1 T + \alpha_2 T^2$, with suitable coefficients $\alpha_1,\alpha_2, R_0 \in \mathbbm{R}$.
The thermally relevant resistors convert electrical power $\bm{p}$ into heat. Joule's law gives
\begin{equation}\label{eq:el-dissipated-power}
 \bm{p} = (p_1,\,\dotsc,\, p_{\nT})^\top= \resistTM(T)\,\operatorname{diag} \bigl(\bm{A}_{\resistT}^{\top} \bm{e}\bigr)  \bm{A}_{\resistT}^{\top} \bm{e}
   = \left(\tfrac{(\bm{A}_{\resistTi{1}}^{\top} \bm{e})^2}{R_1(T_1)},\; \dotsc,\; 
      \tfrac{(\bm{A}_{\resistTi{\nT}}^{\top} \bm{e})^2}{R_{\nT}(T_{\nT})} \right)^{\top}
\end{equation}
where we use the column-wise representation $\bm{A}_{\resistT}=\bigl(\bm{A}_{\resistTi{1}},\dotsc, \bm{A}_{\resistTi{\nT}}\bigr)$.

\subsubsection*{\textbf{Heat evolution model in circuit}} 

For each thermally relevant resistor, let $T_i$ denote the lumped temperature and $M_i$ its heat capacity (heat mass). Heat exchange between thermal elements $i$ and $j$ is described by the conductivity $\Lambda_{i,j}$, with $\Lambda_{i,j}=0$ for unconnected elements. According to Newton cooling each element exchanges heat with an ambient reservoir at temperature $T_{\env}$ with surface coefficient $\Gamma_i$. Heat generation within each element arises from electrical dissipation and is represented by the power input $p_i$ \eqref{eq:el-dissipated-power}. This yields a spatially distributed but lumped thermal model for the temperatures $\bm{z}_T=(T_1,\dotsc, T_{\nT})^{\top}$ that captures both local heat generation and diffusive heat transport. 
Considering surface matrix $\bm{\Gamma}$, heat mass matrix $\bm{M}$, and the symmetric heat-exchange matrix 
$ \bar{\bm{\Lambda}}$ 
$$
\bm{\Gamma}=\operatorname{diag}\bigl(\Gamma_1,\dotsc,\Gamma_{\nT}\bigr), \quad
\bm{M}=\operatorname{diag} \bigl(M_1,\dotsc,M_{\nT} \bigr), \quad
 \bar{\bm{\Lambda}} = \sum_{i < j} 
                \Lambda_{i,j} (-\canonic_i \canonic_i^{\top} + \canonic_i \canonic_j^{\top} + \canonic_j \canonic_i^{\top} - \canonic_j \canonic_j^{\top})
$$
with canonical unit vectors $\canonic_i$, the thermal balance reads, see~\cite{bartel2003},
\begin{equation}\label{eq:class-heat-equation}
  \bm{M} \dot{\bm{z}}_T
  = \bar{\bm{\Lambda}}  \bm{z}_T - \bm{\Gamma} (\bm{z}_T - T_{\env}\one) + \bm{p}
\end{equation}
with $\one=(1,\dotsc, 1)^{\!\top} \!\in \mathbbm{R}^{\nT}$.
To cast the ODE model \eqref{eq:class-heat-equation} into port-Hamiltonian form, we change the variables as in \cite{esterhuizen2024} and introduce entropy $S_i$ which is related to temperature $T_i$ according to
\begin{align*}
    T_i(S_i)= T_{\myref} \exp \left(\tfrac{S_i}{M_i} \right).
\end{align*}
With $\bm{x}_T= (S_1,\dotsc, S_{\nT})^{\top} $, \eqref{eq:class-heat-equation} becomes
\begin{equation}\label{eq:evolution-of-entropy}
     \dot{\bm{x}}_T= 
                \bm{\Lambda}(\bm{x}_T)\bm{z}_T
                - \operatorname{diag} \left(
                                \tfrac{T_1-T_{\env}}{T_1},\, \dotsc,\, \tfrac{T_{\nT}-T_{\env}}{T_{\nT}} 
                              \right){\bm{u}}_T^{\ex} 
                + \operatorname{diag}\left( \tfrac{1}{T_1},\, \dotsc,\, \tfrac{1}{T_{\nT}} \right) \hat{\bm{u}}_T 
\end{equation}
where
$$ \bm{\Lambda} = \sum_{i < j } \tfrac{\Lambda_{ij}\,(T_i-T_j)}{T_i \, T_j} \left(\canonic_j \canonic_i^{\top} - \canonic_i \canonic_j^{\top}\right),
 \qquad \hat{\bm{u}}_T=\bm{p}, \qquad
  {\bm{u}}_T^{\ex} = (\Gamma_1,\dotsc,\Gamma_{\nT})^{\top}. 
$$
The matrix $\bm{\Lambda}$ is skew-symmetric. Hence, heat conduction is represented by the interconnection structure rather than by a dissipative port-Hamiltonian term, while heat exchange with the ambient environment appears as an external port.

\subsubsection*{\textbf{Coupling}}

For the electric subsystem, the temperature-dependent resistors are represented by an internal port $\bm{B_C}(\bm x_C) \hat{\bm{u}}_C$. Then \eqref{eq:MNA-model} can be written as
\begin{align}\label{eq:network-pHS}
\begin{split}
 \begin{pmatrix}
  \bm{A}_C \capM \bm{A}_C^{\top} &  \\
                 & \bm{L}\\
                 & & \bm{0}
  \end{pmatrix} \dot{\bm{x}}_C
  = 
  &
    \left( 
   	\begin{pmatrix}
   		\bm{0}   & -\bm{A}_L & -\bm{A}_V \\ 
		\bm{A}_L^\top &   \bm{0} \\
		\bm{A}_V^\top &      & \bm{0}
	\end{pmatrix} 
    -
   	\begin{pmatrix} \bm{A}_{\resistN} \resistNM \bm{A}_{\resistN}^{\top} & \\ 
   		                           & \bm{0} \\
   		                           &   & \bm{0}
   	\end{pmatrix} 
   	\right) \bm{x}_C 
   	\\
   	& \;+ \begin{pmatrix} -\bm{A}_I & \bm{0} \\
   	                              \bm{0}    & \bm{0} \\
   	                              \bm{0}    & -\bm{I}  \end{pmatrix} {\bm{u}}_C^{\ex}
    + \begin{pmatrix}  -\operatorname{diag} \bigl(\bm{A}_{\resistT} \bm{A}_{\resistT}^{\top} \bm{e}\bigr) 
    \\
    \bm{0} 
    \\
    \bm{0}
      \end{pmatrix} \hat{\bm{u}}_C,
\end{split}
\end{align}
with inputs and outputs
\begin{align*}
&\hat{\bm{u}}_C 
= \left( \tfrac{1}{R_1 (T_1)},\; \dotsc,\; \tfrac{1}{R_{\nT} (T_{\nT})}  \right)^{\!\!\top}\!, 
                    &&\hat{\bm{y}}_C = \bm{B}_C(\bm x_C)^{\top} \bm{z}_C 
                 = -\left( (\bm{A}_{\resistTi{1}}^{\top} \bm{e})^2,\, \dotsc,\,  (\bm{A}_{\resistTi{\nT}}^{\top} \bm{e})^2\right)^\top,\\
&\bm{u}_C^{\ex} = \begin{pmatrix} \bm{\imath} \\ \bm{v}  \end{pmatrix}\!, 
                   &&\bm{y}_C^{\ex} = \Bigl(\bm{B}_C^{\ex}\Bigr)^{\!\!\top} \bm{x}_C = \begin{pmatrix} -\bm{A}_I^{\top} \bm{e} \\ -\bm{A}_V \bm{\jmath}_V \end{pmatrix}\!.                  
\end{align*}
The thermal subsystem  \eqref{eq:evolution-of-entropy} has the inputs and outputs
\begin{align*}
& \hat{\bm{u}}_T=\bm p=\left(\tfrac{(\bm{A}_{\resistTi{1}}^{\top} \bm{e})^2}{R_1(T_1)},\; \dotsc,\; 
      \tfrac{(\bm{A}_{\resistTi{\nT}}^{\top} \bm{e})^2}{R_{\nT}(T_{\nT})} \right)^{\top}, 
    &&\hat{\bm{y}}_T =\bm{B}_T(\bm x_T)^\top \bm{z}_T = \one \in \mathbbm{R}^{\nT}, \\
& {\bm{u}}_T^{\ex} = (\Gamma_1,\dotsc,\Gamma_{\nT})^{\top}, 
    &&\bm{y}_T^{\ex} =(\bm{B}_T^{\ex}(\bm x_T))^\top  \bm{z}_T = - (T_1\!-\!T_{\env},\; \dotsc,\; T_{\nT}\!-\!T_{\env})^{\!\top}.
\end{align*}
The two subsystems are interconnected through the skew-symmetric relation \cite{ehrhardt2026}
\begin{align}\label{eq:pH-coupling-general-electro-thermal}
	\begin{pmatrix}
		\hat{\bm{u}}_C \\ \hat{\bm{u}}_T
	\end{pmatrix}
	= 
	 \begin{pmatrix} \bm{0} &  \hat{\bm{C}}\\
		 -\hat{\bm{C}}^{\top} & \bm{0}\end{pmatrix}
	\begin{pmatrix}
		\hat{\bm{y}}_C \\ \hat{\bm{y}}_T
	\end{pmatrix}
    \quad 
    \text{with}
    \quad 
    \hat{\bm{C}} \coloneqq \resistTM = \text{diag}\left( \tfrac{1}{R_1 (T_1)},\; \dotsc,\; \tfrac{1}{R_{\nT} (T_{\nT})} 
                   \right).
\end{align}
Thus, the electrical power dissipated by the thermally relevant resistors enters the thermal subsystem exactly as the corresponding heat input, while the temperature dependence of the electrical resistance is retained in the electrical subsystem.
The Hamiltonians of the electric and thermal subsystems are
\[
 \mathcal{H}_C(\bm{x}_C) = \frac{1}{2} \bm{x}_C^{\top} \begin{pmatrix}
       \bm{A}_C \bm{C}_C \bm{A}_C^{\top} &  \\
                  & \bm{L}\\
                  & & \bm{0}
 \end{pmatrix} \bm{x}_C , 
 \qquad
   \mathcal{H}_T(\bm{x}_T) 
                = \sum_{i=1}^{\nT} M_i \,T_i(S_i).
\]
Hence, the total stored energy is the sum of the electrical and thermal contributions.

\subsection{Scalable Benchmark Problem} \label{subsec: Scalable ODE}

As a benchmark for the numerical simulations in Section~\ref{sec:numerical_results}, we consider the scalable electro-thermal system shown in Fig.~\ref{fig:ODE-thermal-electric-scalable}. The basic building block is the electro-thermal network enclosed in the dashed box. Each block comprises a temperature-dependent resistor \(R_i (T_i)\) and a parallel $RLC$-link to ground. The $i$th-block is described by node potential $e_i$, the inductor current $\jmath_{i}$, and the lumped entropy $S_i$ with temperature $T_i(S_i)$. The network is driven by a current source \(\imath(t)\) at node $e_0$.
Connecting \(\nbuild\) identical blocks in a chain gives a scalable model with $n_C=2N+1$ electrical and $n_T=N$ thermal state variables.
The state $\bm x=(\bm x_C^\top, \bm x_T^\top)^\top$ and effort variables $\bm z(\bm x)=(\bm z_C^\top, \bm z_T^\top)^\top$ 
\[
   \bm{x}_C= \bm{z}_C = (e_0,e_1,\jmath_1,e_2,\jmath_2,\dotsc, e_{\nbuild},\jmath_{\nbuild})^{\top} , \qquad
   \bm{x}_T = (S_1,\dotsc, S_{\nbuild})^\top, \quad  \bm{z}_T = (T_1,\dotsc, T_{\nbuild})^\top,
\]   
satisfy the coupled pH-ODE, $\bm{x}(t_0)=\bm{x}_0$,
\begin{align} \label{eq:benchmark}
\begin{pmatrix}
        \bm{E}_C & \bm{0}\\
        \bm{0}   & \bm{I}
    \end{pmatrix} \dot{\bm x}
    &=(\begin{pmatrix}
        \bm{J}_C & {\bm{C}}(\bm{x})\\[0.3em]
        - {\bm{C}}(\bm{x})^{\top} & \bm{J}_T(\bm{x}_T)
        \end{pmatrix}
        -\begin{pmatrix}
        \bm{R}_C\! & \bm{0}\\
        \bm{0}        & \bm{0}
    \end{pmatrix}) \, \bm z(\bm x) + 
    \begin{pmatrix}
        \bm{B}_C^{\ex} & \bm{0}\\
        \bm{0} & \bm{B}_T^{\ex}(\bm{x}_T)
    \end{pmatrix} \bm u(t),\\
    \bm{y} &= \bm{B}(\bm{x})^{\top} \bm{z}(\bm{x}), \nonumber
\end{align}
with coupling matrix ${\bm C}(\bm x)=-{\bm B}_C(\bm x_C) \hat{\bm C}(\bm x) {\bm B}_T(\bm x_T)^\top $ as induced by \eqref{eq:evolution-of-entropy}, \eqref{eq:network-pHS}, and \eqref{eq:pH-coupling-general-electro-thermal}.
In particular, the subsystem matrices for the inner dynamics are
\begin{align*}
&\bm J_C=\begin{pmatrix}
		0 &   &   &  \\
		  & 0 & -1 &   \\
		  & 1 &  0 &  \\
		  &   &    &  \ddots \\
		  &   &    &         & 0 & -1  \\
		  &   &    &         & 1 &  0  \\
	\end{pmatrix}, 
	\quad
	\bm R_C = \operatorname{diag}(0,\tfrac{1}{R},0,\dots,\tfrac{1}{R},0),
	\quad
	\bm E_C = \operatorname{diag}(C_0,C,L,\dots,C,L),\\
	&\bm J_T(\bm x_T)= \begin{pmatrix} 0 & \hspace*{0ex}-\frac{\Lambda_{1,2}(T_{1}-T_{2})}{T_{1} T_2} \hspace*{4ex}&  \\[1ex]
		\frac{\Lambda_{1,2}(T_{1}-T_{2})}{T_{1} T_{2}} & \hspace*{-2.5ex} 0 \hspace*{0.5ex} & \hspace*{4ex}-\frac{\Lambda_{2,3}(T_{2}-T_{3})}{T_{2} T_{3}} \\[2ex]
		& \hspace*{-2ex}\ddots &\hspace*{-6ex}\ddots & \hspace*{-5ex}\ddots \\[3ex]
		&        & \hspace*{-10ex}\frac{\Lambda_{\nbuild-2,\nbuild-1}(T_{\nbuild-2}-T_{\nbuild-1})}{T_{\nbuild-2} T_{\nbuild-1}}\hspace*{-1ex}  &  \hspace*{-2ex}0     &  -\frac{\Lambda_{\nbuild-1,\nbuild}(T_{\nbuild-1}-T_{\nbuild})}{T_{\nbuild-1} T_{\nbuild}} \\[1ex]
		&        &   & \hspace*{-1ex}\frac{\Lambda_{\nbuild-1,\nbuild}(T_{\nbuild-1}-T_{\nbuild})}{T_{\nbuild-1} T_{\nbuild}}\hspace*{-1ex} & 0
	\end{pmatrix};
\end{align*}
the external port matrices are
\begin{align*}	
	&\bm B_C^{\ex} = (1,0,\dots,0)^\top \in \mathbb{R}^{ n_C \times 1}, \qquad
	\bm B_T^{\ex}(\bm x_T) = \operatorname{diag}(
		-(1-\tfrac{T_{\env}}{T_{1}}), \dots,-(1-\tfrac{T_{\env}}{T_{\nbuild}})) \in \mathbb{R}^{ n_T \times n_T}
\end{align*}
with external input
 $\bm{u} = ( (\bm{u}_C^{\ex})^\top,  (\bm{u}_T^{\ex})^\top)^\top=(\imath, (\Gamma_1, \dots \Gamma_N))^\top$; and
the matrices for the internal coupling ${\bm C}$ are
\begin{align*}
&\bm{B}_C(\bm{x}_C)=\begin{pmatrix} 
    		-(e_0-e_1) \\[0.5ex]
    		 \phantom{-}(e_0-e_1) &	-(e_1-e_2) \\
    		         0  &      0\\
            		   &            & \ddots &  \\
     		           &            && \phantom{-}(e_{\nbuild-2}-e_{\nbuild-1}) &	-(e_{\nbuild-1}-e_{\nbuild}) \\
     		           &            &&                   0          &     0 \\
    		           &            &&                              & (e_{\nbuild-1}-e_{\nbuild}) 	 \\            		
     		           &            &&                              &    0         \\
    \end{pmatrix}\in \mathbb{R}^{n_C\times n_T},\\
    \quad 
& \bm{B}_T(\bm{x}_T)=\operatorname{diag}(\tfrac{1}{T_1},\dots,\tfrac{1}{T_N})\in \mathbb{R}^{n_T\times n_T},
    \qquad
    \hat{\bm{C}}(\bm x)=\operatorname{diag}(\tfrac{1}{R_1(T_1)},\dots,\tfrac{1}{R_N(T_N)}).
\end{align*}
Consequently, the internal inputs and outputs are 
\begin{align*}
  & \hat{\bm{u}}_C =\left(\tfrac{1}{R_1 (T_1)}, \dotsc, \tfrac{1}{R_{\nbuild}(T_{\nbuild})} \right)^{\top} , 
   && \hat{\bm{y}}_C = \left( -(e_0\!-\!e_1)^2, \dotsc, - (e_{\nbuild-1}\!-\!e_\nbuild)^2 \right)^{\top}, \\
  & \hat{\bm{u}}_T = \left( \tfrac{(e_0-e_1)^2}{R_1(T_1)}, \dotsc, 
		   \tfrac{(e_{\nbuild-1}-e_\nbuild)^2}{R_{\nbuild}(T_\nbuild)}   \right)^{\top},
               && \hat{\bm{y}}_T = \one \in \R^{n_T}.           
\end{align*}
The total Hamiltonian is given by 
$$\mathcal{H}(\bm x)=\mathcal{H}_C(\bm x_C)+\mathcal{H}_T(\bm x_T)=
\frac{1}{2}\left( C_0 e_0^2 + \sum_{i=1}^N C e_i^2 +  L\jmath_i^2 \right)
+ \sum_{i=1}^{N}M_i T_i(S_i). $$
Circuit and thermal parameters used in the numerical simulations are specified in Section~\ref{sec:benchmark}. 

\begin{remark} 
\begin{itemize}
\item[i)] The general pH-DAE circuit description \eqref{eq:network-pHS} simplifies to the pH-ODE \eqref{eq:benchmark}, first subsystem, for the scalable benchmark shown in Figure~\ref{fig:ODE-thermal-electric-scalable}. This is due to the absence of voltage sources and the fact that every (electric) node is connected to ground via a capacitor.
\item[ii)] Additional intermediate temperature states can be introduced between \(T_i\) and \(T_{i+1}\), for instance to represent a discretized thermal substrate. This refinement increases the dimensionality of the thermal system while retaining a lumped-parameter description.
\end{itemize}
\end{remark}

\begin{figure}[bt]
         \includegraphics{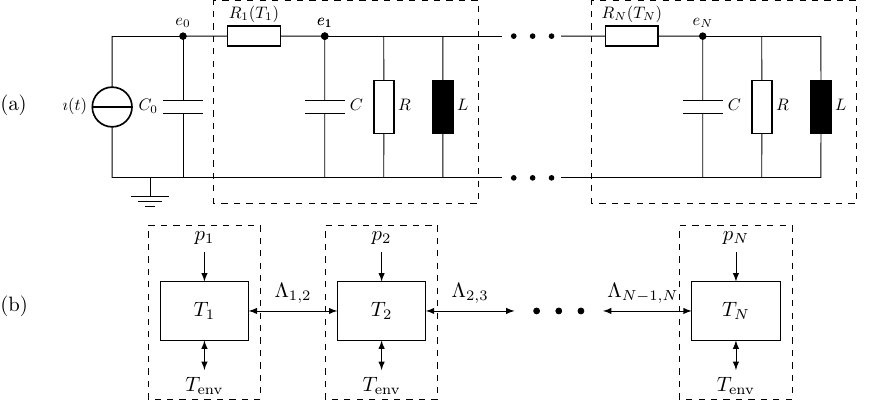}
	\caption{Scalable ODE electro-thermal circuit problem with $\nbuild$ building blocks in the dashed box. (a) Electric schematic, (b) thermal schematic.}
	\label{fig:ODE-thermal-electric-scalable}
\end{figure}

\bibliographystyle{abbrv}

\bibliography{Quellen}

\end{document}